\documentclass[titlepage]{article}

\usepackage[utf8]{inputenc} % set input encoding (not needed with XeLaTeX)

\usepackage{geometry} % to change the page dimensions
\usepackage{graphicx} % support the \includegraphics command and options

\usepackage{booktabs} % for much better looking tables
\usepackage{array} % for better arrays (eg matrices) in maths
\usepackage{paralist} % very flexible & customisable lists (eg. enumerate/itemize, etc.)
\usepackage{verbatim} % adds environment for commenting out blocks of text & for better verbatim
\usepackage{subfig} % make it possible to include more than one captioned figure/table in a single float
\usepackage{fancyhdr} % This should be set AFTER setting up the page geometry
\usepackage{sectsty}
\allsectionsfont{\rmfamily\mdseries\upshape} % (See the fntguide.pdf for font help)
\usepackage[nottoc,notlof,notlot]{tocbibind} % Put the bibliography in the ToC
\usepackage[titles,subfigure]{tocloft} % Alter the style of the Table of Contents

\usepackage{amssymb}
\usepackage{amsmath}
\usepackage{amsthm}
\usepackage{bbm}

\DeclareMathOperator{\trace}{tr}
\DeclareMathOperator{\a.s.}{a.s.}
\DeclareMathOperator{\ess}{ess}
\DeclareMathOperator{\dist}{dist}
\DeclareMathOperator{\image}{image}

\numberwithin{equation}{section}

\newtheorem{thm}{Theorem}[section]
\newtheorem{lem}[thm]{Lemma}
\newtheorem{prop}[thm]{Proposition}
\newtheorem{cor}[thm]{Corollary}
\theoremstyle{definition}
\newtheorem{defn}[thm]{Definition}
\newtheorem{exmp}[thm]{Example}
\theoremstyle{remark}
\newtheorem{rem}[thm]{Remark}

\let\phi\varphi
\let\epsilon\varepsilon

\def\Pbb{\mathbb{P}}
\def\Ebb{\mathbb{E}}
\def\Rbb{\mathbb{R}}
\def\Fbb{\mathbb{F}}
\def\Nbb{\mathbb{N}}
\def\Qbb{\mathbb{Q}}
\def\1bb{\mathbbm{1}}
\def\Fcal{\mathcal{F}}
\def\Lcal{\mathcal{L}}
\def\Hcal{\mathcal{H}}
\def\Pcal{\mathcal{P}}
\def\Bcal{\mathcal{B}}
\def\Ecal{\mathcal{E}}
\def\Dcal{\mathcal{D}}
\def\Acal{\mathcal{A}}
\def\Hfrak{\mathfrak{H}}
\def\pfrak{\mathfrak{p}}
\def\Pfrak{\mathfrak{P}}

\begin{document}
\title{\Huge Backward Stochastic Differential Equations in Financial Mathematics}
\author{\huge Weiye Yang}
\date{} % Activate to display a given date or no date (if empty),
         % otherwise the current date is printed 
\maketitle

\tableofcontents

\newpage
\setcounter{section}{-1}
\section{Introduction}

A backward stochastic differential equation (BSDE) is an SDE of the form
\begin{align*}
 -dY_t &= f(t,Y_t,Z_t)dt - Z_t^*dW_t;\\
 Y_T &= \xi.
\end{align*}
It differs from a forward stochastic differential equation (FSDE) in two main aspects:
\begin{itemize}
 \item A BSDE has a {\em terminal} condition $Y_T = \xi$, as opposed to FSDEs, which have initial conditions.
 \item A solution of a BSDE consists of a {\em pair of processes} $(Y,Z)$ satisfying the equation.
\end{itemize}
The subject of BSDEs has seen extensive attention since their introduction in the linear case by Bismut (1973) and in the general case by Pardoux and Peng (1990). In contrast with deterministic differential equations, it is not enough to simply reverse the direction of time and treat the terminal condition as an initial condition, as we would then run into problems with adaptedness. Intuitively, our ``knowledge'' at time $t$ consists only of what has happened at all times $s \in [0,t]$, and we cannot reverse the direction of time whilst keeping this true. In this way, the theory of BSDEs is very much its own beast, separate from the theory of FSDEs.

BSDEs have a number of applications in mathematical finance. For example, we can see immediately that the problem of finding the time-$t$ price of a European contingent claim expiring at some fixed future time $T \ge t$ is exactly the problem of solving a linear BSDE. Additionally there is an intimate link between BSDEs and partial differential equations, which allows one to express the solution of a BSDE, and thus the price of a European contingent claim, in terms of the solution of a related PDE. This gives a generalisation of the Black-Scholes formula. BSDEs also act as the ``Euler-Lagrange equations'' in a number of utility maximisation problems, in the sense that the optimal value of such problems can be found through solving the related BSDE.

The layout of this essay is as follows: In Section \ref{BSDEs} we introduce BSDEs and go over the basic results of BSDE theory, including two major theorems: the existence and uniqueness of solutions and the comparison theorem. We also introduce linear BSDEs and the notion of supersolutions of a BSDE. In Section \ref{Eurodyncomp} we set up the financial framework in which we will price European contingent claims, and prove a result about the fair price of such claims in a dynamically complete market. In Section \ref{concave} we extend the theory of linear BSDEs to include concave BSDEs, and apply this to pricing claims in more complicated market models. In Section \ref{utility} we take a look at utility maximisation problems, and see how utilising BSDE theory allows for a relatively simple and neat solution in certain cases.

All proofs, unless stated otherwise, are my own.
\newpage

\section{Backward Stochastic Differential Equations}\label{BSDEs}

\subsection{Preliminary notation}
Before we start talking about SDEs we need a probability space on which to work. Here is our setup. Fix a finite time horizon $T>0$, and then:
\begin{itemize}
 \item Let $(\Omega,\Fcal,\Pbb)$ be a probability space in which we have an $n$-dimensional Brownian motion $W$.
 \item Let $\Fbb = (\Fcal_t)_{t \in [0,T]}$ be the augmented (i.e. right-continuous and containing $\Pbb$-null sets) filtration generated by $W$ up to time $T$.
\end{itemize}
Note that this means that all solutions to SDEs in this setup will be {\em strong} solutions. It also allows us to use the martingale representation theorem freely. We now define some notation.
\begin{itemize}
 \item For vectors $a,b \in \Rbb^d$, let $|a|$ denote the standard Euclidean norm and $a \cdot b$ the standard Euclidean inner product.
 \item For a matrix $A \in \Rbb^{n \times d}$, let $A^* \in \Rbb^{d \times n}$ denote its transpose and $|A|$ its Frobenius norm, i.e. $|A|^2 = \trace (AA^*)$.
 \item For continuous semimartingales $X = (X_t)_{t \in [0,T]}$, $Y = (Y_t)_{t \in [0,T]}$ taking values in $\Rbb$, $\Rbb ^d$ respectively, let $\langle X \rangle = (\langle X \rangle_t)_{t \in [0,T]}$ denote the quadratic variation and $\langle X,Y \rangle = (\langle X,Y \rangle_t)_{t \in [0,T]}$ the $\Rbb ^d$-valued quadratic covariation, where $\langle X,Y \rangle^i = \langle X,Y^i \rangle$ which is the standard $\Rbb$-valued quadratic covariation.
 \item Let $\Lcal^{2,d}_T$ be the space of square-integrable $\Fcal_T$-measurable random variables taking values in $\Rbb^d$, i.e. the space of $\Fcal_T$-measurable $X: \Omega \to \Rbb^d$ such that $\lVert X \rVert ^2 := \Ebb\left[ |X|^2 \right] < \infty$.
 \item Let $\Hcal^{2,d}_T$ be the space of predictable processes $\phi : \Omega \times [0,T] \to \Rbb^d$ such that $\lVert \phi \rVert ^2 := \Ebb\left[ \int_0^T |\phi_t|^2 dt \right] < \infty$. We call these processes {\em square-integrable}.
 \item Let $\Hcal^{1,d}_T$ be the space of predictable processes $\phi : \Omega \times [0,T] \to \Rbb^d$ such that $\Ebb\left[ \sqrt{\int_0^T |\phi_t|^2 dt} \right] < \infty$.
 \item For $\beta > 0$ and $\phi \in \Hcal^{2,d}_T$, define $\lVert \phi \rVert_\beta ^2 = \Ebb\left[ \int_0^T e^{\beta t}|\phi_t|^2 dt \right]$ and let $\Hcal^{2,d}_{T,\beta}$ be the space $\Hcal^{2,d}_T$ endowed with this norm.
\end{itemize}
\begin{rem}
Recall that $\Lcal^{2,d}_T$ and $\Hcal^{2,d}_T$ are Hilbert spaces. This implies by equivalence of norms that $\Hcal^{2,d}_{T,\beta}$ is also a Hilbert space for any $\beta > 0$.
\end{rem}

\subsection{Basic definitions}
\begin{defn}\label{BSDE}
A {\em backward stochastic differential equation} (or {\em BSDE} for short) is a stochastic differential equation of the form
\begin{equation}\label{genBSDE}
\begin{split}
 -dY_t &= f(t,Y_t,Z_t)dt - Z_t^*dW_t;\\
 Y_T &= \xi,
\end{split}
\end{equation}
or equivalently
\begin{equation}\label{BSDEint}
Y_t = \xi + \displaystyle\int_t^T f(s,Y_s,Z_s)ds - \displaystyle\int_t^T Z_s^*dW_s,
\end{equation}
where
\begin{enumerate}
 \item the random variable $\xi: \Omega \to \Rbb ^d$ is $\Fcal_T$-measurable,
 \item the random function $f: \Omega \times [0,T] \times \Rbb^d \times \Rbb^{n \times d} \to \Rbb^d$ (known as the {\em driver}) is $\Pfrak \otimes \Bcal^d \otimes \Bcal^{n \times d}$-measurable, where $\Pfrak$ is the predictable $\sigma$-algebra over $\Omega \times [0,T]$.
\end{enumerate}
The pair $(f,\xi)$ is known as the {\em data} of the BSDE.
\end{defn}

Notice that there are {\em two} unknown processes in a BSDE: $Y$ and $Z$. It is therefore worth giving an explicit definition of what we mean by a solution of a BSDE.

\begin{defn}
A continuous {\em solution} of BSDE \eqref{genBSDE} is a pair $(Y,Z) = (Y_t,Z_t)_{t \in [0,T]}$ such that $Y$ is a continuous adapted $\Rbb^d$-valued process and $Z$ is a predictable $\Rbb^{n \times d}$-valued process with $\int_0^T |Z_t|^2ds < \infty$ $\ $ $\Pbb \ \a.s.$, which satisfies the BSDE. The solution $(Y,Z)$ is {\em square-integrable} if $(Y,Z) \in \Hcal^{2,d}_T \times \Hcal^{2,n \times d}_T$.
\end{defn}

For our first main result we will need to assume some regularity conditions on the data:

\begin{defn}\label{standard}
The driver $f$ is {\em standard} if $f(\cdot,0,0) \in \Hcal^{2,d}_T$ and $f(\omega,t,y,z)$ is uniformly Lipschitz in $(y,z)$. This latter property means that $\exists C > 0$ such that $d\Pbb \otimes dt \ \a.s.$,
\begin{equation*}
|f(\cdot,y_1,z_1) - f(\cdot,y_2,z_2)| \leq C\left( |y_1 - y_2| + |z_1 - z_2| \right) \quad \forall (y_1,z_1),(y_2,z_2) \in \Rbb^d \times \Rbb^{n \times d}.
\end{equation*}
If in addition $\xi \in \Lcal^{2,d}_T$, we say that the data $(f,\xi)$ are {\em standard} data.
\end{defn}

It's worth making the following small observation before we move on:

\begin{prop}\label{fisHcal2}
Suppose that $f$ is a standard driver with Lipschitz constant $C$ and that $(y,z) \in \Hcal^{2,d}_T \times \Hcal^{2,n \times d}_T$. Then $f(\cdot,y,z) \in \Hcal^{2,d}_T$.
\end{prop}

\begin{proof}
Directly from Definition \ref{standard}, we have that $d\Pbb \otimes dt \ \a.s.$,
\begin{equation*}
|f(\cdot,y,z) - f(\cdot,0,0)| \leq C\left( |y| + |z| \right).
\end{equation*}
Hence
\begin{equation*}
|f(\cdot,y,z))| \leq |f(\cdot,0,0))| + C\left( |y| + |z| \right)
\end{equation*}
and the result follows by squaring and integrating.
\end{proof}

\subsection{Existence and uniqueness of solutions}
The purpose of this subsection is to prove the following theorem:

\begin{thm}[Pardoux--Peng (1990)]\label{ExistUniq}
If $(f,\xi)$ are standard data, then there exists a unique continuous square-integrable solution to BSDE \eqref{genBSDE}.
\end{thm}

The method of proof is similar to other existence-uniqueness theorems in SDEs: we find a suitable complete space and a suitable mapping from that space into itself, and then invoke the contraction mapping theorem. To this end, recall the following result from the classical theory:

\begin{prop}[Burkholder--Davis--Gundy inequalities]\label{BDG}
For any $p > 0$ there exist positive constants $c_p,C_p$ such that, for all real-valued continuous local martingales $X$ with $X_0 = 0$ and stopping times $\tau$, the following inequality holds:
\begin{equation}
c_p\Ebb\left[ \langle X \rangle^{p/2}_\tau \right] \leq \Ebb\left[ (X^*_\tau)^p \right] \leq C_p\Ebb\left[ \langle X \rangle^{p/2}_\tau \right]
\end{equation}
where $X^*_t = \sup_{s \leq t}|X_s|$ is the maximum process of $X$.
\end{prop}

A proof of this result can be found in Karatzas and Shreve (1991). We prove a corollary that will be useful later:

\begin{cor}\label{unifmart}
Let $\phi \in \Hcal^{1,n}_T$, and define $M_t = \int_0^t \phi_s \cdot dW_s$ for $t \in [0,T]$. Then $(M_t)_{t \in [0,T]}$ is a uniformly integrable martingale.
\end{cor}

\begin{proof}
$M$ is a real-valued continuous local martingale with $M_0 = 0$. By the Burkholder--Davis--Gundy inequalities with $p = 1$ and $\tau = T$,
\begin{align*}
\Ebb\left[ \sup_{s \leq T}|M_s| \right] &\leq C_1\Ebb\left[ \langle M \rangle^{1/2}_T \right]\\
 &= C_1\Ebb\left[ \left( \int_0^T |\phi_s|^2 ds \right) ^\frac{1}{2} \right]\\
 & < \infty.
\end{align*}
Thus the process $M$ is dominated by the integrable random variable $\sup_{s \leq T}|M_s|$. So, using dominated convergence, we find that it is a uniformly integrable martingale as required.
\end{proof}

In order to prove the theorem we need to find a mapping with a suitable Lipschitz constant. To do this we first need to make a few apriori estimates that control the size of the difference between two solutions. They may look nasty but their eventual application will be very straightforward. First, a lemma.

\begin{lem}\label{supLcal}
Suppose BSDE \eqref{genBSDE} has a square-integrable solution $(Y,Z)$ with standard data $(f,\xi)$. Then $\sup_{t\leq T}|Y_t| \in \Lcal^{2,1}_T$.
\end{lem}

\begin{proof}[Proof from El Karoui (1997)]
From equation \eqref{BSDEint}, we get
\begin{align*}
|Y_t| &= \left| \xi - \int_t^T f(s,Y_s,Z_s)ds - \int_t^T Z_s^*dW_s \right|\\
 &\leq |\xi| + \int_t^T |f(s,Y_s,Z_s)|ds + \left| \int_t^T Z_s^*dW_s \right|\\
 &\leq |\xi| + \int_0^T |f(s,Y_s,Z_s)|ds + \sup_{t\leq T} \left| \int_t^T Z_s^*dW_s \right|
\end{align*}
and hence
\begin{equation}\label{existuniqabs}
\sup_{t\leq T}|Y_t| \leq |\xi| + \displaystyle\int_0^T |f(s,Y_s,Z_s)|ds + \sup_{t\leq T} \left| \int_t^T Z_s^*dW_s \right|.
\end{equation}
By the Burkholder-Davis Gundy inequalities (Proposition \ref{BDG}):
\begin{align*}
\Ebb \left[ \sup_{t\leq T} \left| \int_t^T Z_s^*dW_s \right|^2 \right] &\leq 2 \Ebb \left[ \left| \int_0^T Z_s^*dW_s \right|^2 \right] + 2 \Ebb \left[ \sup_{t\leq T} \left| \int_0^t Z_s^*dW_s \right|^2 \right]\\
 &\leq 4 C_2 \Ebb \left[ \int_0^T |Z_s|^2ds \right]\\
 &< \infty
\end{align*}
since $Z$ is square-integrable. By Proposition \ref{fisHcal2}, $f(\cdot,Y,Z) \in \Hcal^{2,d}_T$. Hence every term on the right-hand side of equation \eqref{existuniqabs} is in $\Lcal^{2,1}_T$ and so $\sup_{t\leq T}|Y_t| \in \Lcal^{2,1}_T$.
\end{proof}

\begin{prop}[A priori estimates]\label{Apriori}
Let $((f^i,\xi^i);i=1,2)$ be the standard data of two BSDEs \eqref{genBSDE} and suppose that the BSDEs have square-integrable solutions $((Y^i,Z^i);i=1,2)$ respectively. Let $C$ be a Lipschitz constant for $f^1$, and set $\delta Y = Y^1 - Y^2$, $\delta Z = Z^1 - Z^2$ and $\delta _2f = f^1(\cdot,Y^2,Z^2) - f^2(\cdot,Y^2,Z^2)$. Then for any $\beta > C(2+C)$, the following inequalities hold:
\begin{subequations}\label{apriori}
\begin{align}
\lVert \delta Y \rVert _\beta^2 &\leq T \left( e^{\beta T}\Ebb \left[ |\delta Y_T|^2 \right] + \frac{1}{\beta-2C-C^2} \lVert \delta _2f \rVert_\beta^2 \right),\\
\lVert \delta Z \rVert _\beta^2 &\leq (2+2C^2T) \left( e^{\beta T}\Ebb \left[ |\delta Y_T|^2 \right] + \frac{1}{\beta-2C-C^2} \lVert \delta _2f \rVert_\beta^2 \right).
\end{align}
\end{subequations}
\end{prop}

\begin{proof}[Proof from El Karoui (1997)]
From It\={o}'s formula applied from $s=t$ to $s=T$ to the semimartingale $e^{\beta s}|\delta Y_s|^2$, we get
\begin{equation}\label{ito1}
\begin{split}
e^{\beta t}|\delta Y_t|^2 &= e^{\beta T}|\delta Y_T|^2 + 2\displaystyle\int_t^T e^{\beta s} \delta Y_s \cdot \left( f^1(s,Y^1_s,Z^1_s) - f^2(s,Y^2_s,Z^2_s) \right) ds\\
&\quad -\beta \displaystyle\int_t^T e^{\beta s}|\delta Y_s|^2ds - \displaystyle\int_t^T e^{\beta s}|\delta Z_s|^2ds - 2\displaystyle\int_t^T e^{\beta s}\delta Y_s \cdot \delta Z^*_sdW_s.
\end{split}
\end{equation}
Now $(e^{\beta t}\delta Z_t \delta Y_t)_t \in \Hcal^{1,n}_T$ because using H\"{o}lder's inequality,
\begin{align*}
\Ebb \left[ \sqrt{\displaystyle\int_0^T |e^{\beta t}\delta Z_t \delta Y_t|^2 dt} \right] &\leq \Ebb \left[ \sup\limits_t |\delta Y_t| \sqrt{\displaystyle\int_0^T e^{2\beta t}|\delta Z_t|^2 dt} \right]\\
 &\le \Ebb \left[ \sup\limits_t |\delta Y_t|^2 \right]^\frac{1}{2} \Ebb \left[ \displaystyle\int_0^T e^{2\beta t}|\delta Z_t|^2 dt \right]^\frac{1}{2}\\
 &< \infty
\end{align*}
by Lemma \ref{supLcal} and square-integrability of $\delta Z$. Hence by Corollary \ref{unifmart}, the stochastic integral term in \eqref{ito1} is a uniformly integrable martingale and thus has zero expectation. Using the Lipschitz property of the driver $f^1$ we have that
\begin{align*}
|f^1(s,Y^1_s,Z^1_s) - f^2(s,Y^2_s,Z^2_s)| &\le |f^1(s,Y^1_s,Z^1_s) - f^1(s,Y^2_s,Z^2_s)| + |\delta _2 f_s|\\
 &\le C \left( |\delta Y_s| + |\delta Z_s| \right) + |\delta _2 f_s|.
\end{align*}
So by using the Lipschitz property in equation \eqref{ito1} and taking expectations we get
\begin{equation}\label{ito2}
\begin{split}
\Ebb \left[ e^{\beta t}|\delta Y_t|^2 \right] &\le \Ebb \left[ e^{\beta T}|\delta Y_T|^2 + 2\displaystyle\int_t^T e^{\beta s} |\delta Y_s|  \left( C \left( |\delta Y_s| + |\delta Z_s| \right) + |\delta _2 f_s| \right) ds \right]\\
&\quad + \Ebb \left[ -\beta \displaystyle\int_t^T e^{\beta s}|\delta Y_s|^2ds - \displaystyle\int_t^T e^{\beta s}|\delta Z_s|^2ds \right].
\end{split}
\end{equation}
We notice that a quadratic form appears on the right-hand side of \eqref{ito2}, given by
\begin{equation}
Q(y,z) = 2C|y|^2 + 2C|y||z| + 2|\delta _2 f_s||y| - \beta |y|^2 - |z|^2
\end{equation}
which we can rearrange as
\begin{equation}
Q(y,z) = -\beta _C (|y| - \beta_C^{-1}|\delta _2 f_s|)^2 - (|z| - C|y|)^2 + \beta _C^{-1}|\delta _2 f_s|^2
\end{equation}
where $\beta _C := \beta -2C-C^2 > 0$ by assumption. Hence \eqref{ito2} becomes
\begin{align*}
\Ebb \left[ e^{\beta t}|\delta Y_t|^2 \right] &+ \Ebb \left[ \beta_C \displaystyle\int_t^T e^{\beta s} (|\delta Y_s| - \beta_C^{-1}|\delta_2 f_s|)^2ds + \displaystyle\int_t^T e^{\beta s} (|\delta Z_s| - C|\delta Y_s|)^2ds \right]\\
& \le \Ebb \left[ e^{\beta T}|\delta Y_T|^2 \right] + \Ebb \left[ \displaystyle\int_t^T e^{\beta s} \frac{|\delta_2f_s|^2}{\beta_C}ds \right].
\end{align*}
The target inequality \eqref{apriori} for $\delta Y$ follows directly from the above inequality by integrating between $0$ and $T$. For $\delta Z$, notice $|\delta Z_s|^2 \le 2(|\delta Z_s| - C|\delta Y_s|)^2 + 2C^2|\delta Y_s|^2$ so
\begin{equation*}
\lVert \delta Z \rVert _\beta^2 \le 2\Ebb \left[ \displaystyle\int_0^T e^{\beta s} (|\delta Z_s| - C|\delta Y_s|)^2ds \right] + 2C^2\lVert \delta Y \rVert _\beta^2
\end{equation*}
and the result follows.
\end{proof}

The proof of the main theorem is an immediate consequence of the previous proposition.

\begin{proof}[Proof of Theorem \ref{ExistUniq}]
As may be evident from our previous work, the main task of this proof is to construct a contraction mapping on $\Hcal^{2,d}_{T,\beta} \times \Hcal^{2,n \times d}_{T,\beta}$, for some $\beta > 0$ to be specified later.

Fix $(y,z) \in \Hcal^{2,d}_T \times \Hcal^{2,n \times d}_T$ and let $M$ be the continuous version of the square-integrable martingale $\Ebb \left[ \int_0^T f(s,y_s,z_s)ds + \xi | \Fcal _t \right]$. By the martingale representation theorem (in Karatzas and Shreve) there exists a unique square-integrable process $Z \in \Hcal^{2,n \times d}_T$ for which $M_t = M_0 + \int_0^t Z_s^*dW_s$. Then define $Y$ by $Y_t = M_t - \int_0^t f(s,y_s,z_s)ds$, so that $Y \in \Hcal^{2,d}_T$ by Proposition \ref{fisHcal2}, and $Y$ is continuous. Notice that the pair $(Y,Z)$ thus defined satisfies the equation
\begin{align*}
 -dY_t &= f(t,y_t,z_t)dt - Z_t^*dW_t;\\
 Y_T &= \xi.
\end{align*}
Hence let $\Psi: \Hcal^{2,d}_T \times \Hcal^{2,n \times d}_T \to \Hcal^{2,d}_T \times \Hcal^{2,n \times d}_T$ be the map that sends $(y,z) \mapsto (Y,Z)$ as above.

Now let $(y^1,z^1)$, $(y^2,z^2)$ be two elements of $\Hcal^{2,d}_{T,\beta} \times \Hcal^{2,n \times d}_{T,\beta}$, and let $(Y^i,Z^i) = \Psi(y^i,z^i)$ for $i=1,2$. We seek to apply Proposition \ref{Apriori}. Note that there is some potential for confusion here: in the statement of the proposition we have to set $f^1 = f(s,y^1_s,z^1_s)$, which actually has no explicit dependence on $(Y^1,Z^1)$, since $(y^1,z^1)$ and $(Y^1,Z^1)$ are different processes. Hence we can set the "Lipschitz constant" in the proposition to be $0$. Also $\delta Y_T = \xi - \xi = 0$ and $\delta _2 f_t = f(t,y^1_t,z^1_t) - f(t,y^2_t,z^2_t)$ so we get the inequalities
\begin{align}
\lVert \delta Y \rVert _\beta^2 &\le \frac{T}{\beta} \Ebb \left[ \displaystyle\int_0^T e^{\beta s} |f(s,y^1_s,z^1_s) - f(s,y^2_s,z^2_s)|^2 ds \right],\\
\lVert \delta Z \rVert _\beta^2 &\le \frac{2}{\beta} \Ebb \left[ \displaystyle\int_0^T e^{\beta s} |f(s,y^1_s,z^1_s) - f(s,y^2_s,z^2_s)|^2 ds \right].
\end{align}
Now $f(\omega,t,y,z)$ is uniformly Lipschitz in $(y,z)$ with constant $C$ so
\begin{equation}\label{contraction}
\lVert \delta Y \rVert _\beta^2 + \lVert \delta Z \rVert _\beta^2 \le \frac{2(2+T)C^2}{\beta}\left( \lVert \delta y \rVert _\beta^2 + \lVert \delta z \rVert _\beta^2 \right)
\end{equation}
and therefore, if we pick $\beta$ such that $\beta > 2(2+T)C^2$, $\Psi$ is a contraction on the space $\Hcal^{2,d}_{T,\beta} \times \Hcal^{2,n \times d}_{T,\beta}$. Hence by the contraction mapping theorem there exists a unique fixed point of $\Psi$, which is the unique continuous square-integrable solution of the BSDE.
\end{proof}

\begin{rem}
In what sense exactly is our solution ``unique''? What we have proven above is this: The BSDE \eqref{genBSDE} with standard data has a continuous square-integrable solution, denoted by $(Y^1,Z^1)$, say. If $(Y^2,Z^2)$ is another square-integrable solution to the BSDE, then
\begin{equation}\label{solnnorm}
\lVert Y^1 - Y^2 \rVert ^2 \equiv \Ebb\left[ \int_0^T |Y^1_t - Y^2_t|^2 dt \right] = 0,
\end{equation}
which is equivalent to the statement $Y^1 = Y^2$ $d\Pbb \otimes dt$ a.s.. A similar statement holds for $Z^1$ and $Z^2$.

We can, however, do better than this if $Y^2$ is also continuous. Equation \eqref{solnnorm} is equivalent to
\begin{equation*}
\int_0^T |Y^1_t - Y^2_t|^2 dt = 0 \quad \Pbb \ \a.s.,
\end{equation*}
which means that there exists an almost sure set $\Omega_0 \subseteq \Omega$ on which
\begin{equation}\label{almostsureset}
\int_0^T |Y^1_t(\omega) - Y^2_t(\omega)|^2 dt = 0 \quad \forall \omega \in \Omega_0,
\end{equation}
and on which $|Y^1_t(\omega) - Y^2_t(\omega)|$ is continuous in $t$. Fix an $\omega \in \Omega_0$. The integrand in \eqref{almostsureset} above is non-negative and continuous in $t$, and its integral is zero. It therefore must be zero {\em everywhere}. From this we get the following stronger formulation of uniqueness for continuous solutions:
\end{rem}

\begin{cor}
Let $(Y^1,Z^1)$ and $(Y^2,Z^2)$ be continuous square-integrable solutions of the BSDE \eqref{genBSDE} with standard data. Then
\begin{equation}
Y^1_t = Y^2_t \quad \forall t \in [0,T] \ \Pbb \ \a.s..
\end{equation}
\end{cor}

Equivalently we write $Y^1 = Y^2$ $\Pbb$ a.s. for the above statement. With the main proof complete, we now give an example illustrating that the condition that the solution be square-integrable is necessary for uniqueness to hold:

\begin{exmp}[From El Karoui (1997)]\label{Dudley}
From Dudley (1977) it is possible, for any finite time horizon $T>0$, to construct an $\Rbb ^d$-valued stochastic integral $I_t := I_0 + \int_0^t \psi_s^*dW_s$ such that $I_0=1$ the vector of ones, $I_T=0$ the vector of zeros, and $\int_0^T \lVert \psi_s \rVert ^2 ds < \infty$, $\Pbb$ a.s.. Note that we cannot have $\psi \in \Hcal^{2,n \times d}_T$ because this would imply that $I$ is a martingale, which it cannot be. Now let the BSDE \eqref{genBSDE} have standard data $(0,\xi)$ and let $(Y,Z)$ be its square-integrable solution, so that $Y_t = \Ebb \left[ \xi | \Fcal _t \right]$ and $Z$ is given by the martingale representation theorem. Then for any $\lambda \in \Rbb$, $(Y+\lambda I,Z +\lambda \psi)$ is also a solution of this BSDE.
\end{exmp}

Note that henceforth, when we refer to "the solution" of a BSDE with standard data, we mean its unique continuous square-integrable solution. The rest of this section will be devoted to proving a few select results about the behaviour of BSDEs.

\subsection{LBSDEs and the comparison theorem}
In this subsection we investigate a special class of BSDEs, from which we then prove a very important and powerful result about general BSDEs.

\subsubsection{Linear backward stochastic differential equations}
\begin{defn}
A {\em linear backward stochastic differential equation} (or {\em LBSDE} for short) is a BSDE \eqref{genBSDE} of the form
\begin{equation}\label{genLBSDE}
\begin{split}
 -dY_t &= (\phi_t + Y_t\beta_t + Z_t^*\gamma_t)dt - Z_t^*dW_t;\\
 Y_T &= \xi.
\end{split}
\end{equation}
where $\phi$, $\beta$ and $\gamma$ are progressively measurable $\Rbb ^d$-, $\Rbb$- and $\Rbb ^n$-valued processes respectively.
\end{defn}

LBSDEs are well-known in mathematical finance, as the classical problem of pricing a European contingency claim takes the form of an LBSDE, which we will deal with in the next section. For now, some notation: for a continuous semimartingale $M$, we write $\Ecal (M)$ for its {\em stochastic exponential}, i.e.
\begin{equation}
\Ecal (M) := e^{M - \frac{1}{2} \langle M \rangle}.
\end{equation}
We have that $d\Ecal (M)_t = \Ecal (M)_tdM_t$. Evidently if $M$ is a local martingale then $\Ecal (M)$ is also a local martingale. Recall also the following result from the classical theory:

\begin{prop}[Novikov's condition]
Suppose that $M$ is a continuous local martingale and that
\begin{equation}
\Ebb \left[ e^{\frac{1}{2} \langle M \rangle_\infty} \right] < \infty.
\end{equation}
Then $\Ecal (M)$ is a uniformly integrable martingale.
\end{prop}

The proof can be found in Karatzas and Shreve (1991). Note that, due to our finite time horizon, we can replace the $\langle M \rangle_\infty$ in the condition by $\langle M \rangle_T$ whenever we want to use it.

\begin{defn}
Consider the LBSDE \eqref{genLBSDE} in the case that $\beta$ and $\gamma$ are bounded. The {\em adjoint process} of the LBSDE is then the $\Rbb$-valued positive process
\begin{equation}
\Gamma := \Ecal \left( \int \beta_s ds + \int \gamma_s \cdot dW_s \right),
\end{equation}
so that $\Gamma$ satisfies the FSDE
\begin{equation}
\begin{split}
d\Gamma_t &= \Gamma_t (\beta_t dt + \gamma_t \cdot dW_t);\\
\Gamma_0 &= 1.
\end{split}
\end{equation}
\end{defn}

\begin{lem}[Properties of the adjoint process]\label{Gammareg}
For $\beta$,$\gamma$ bounded and $\Gamma$ defined above,
\begin{enumerate}
 \item $\Gamma \in \Hcal^{2,1}_T$,
 \item $\sup_{t \le T} \Gamma_t \in \Lcal^{2,1}_T$.
\end{enumerate}
\end{lem}

\begin{proof}
\begin{equation*}
\Gamma_t = \exp \left\{ \int_0^t (\beta_s - \frac{1}{2} |\gamma_s|^2)ds + \int_0^t \gamma_s \cdot dW_s \right\}
\end{equation*}
so
\begin{align*}
(\Gamma_t)^2 &= \exp \left\{ \int_0^t (2\beta_s - |\gamma_s|^2)ds + \int_0^t 2\gamma_s \cdot dW_s \right\}\\
 &= \exp \left\{ \int_0^t (2\beta_s + |\gamma_s|^2)ds \right\} \Ecal \left( \int 2\gamma_s \cdot dW_s \right)_t.
\end{align*}
By Novikov's condition and the boundedness of $\gamma$, $\Ecal \left( \int 2\gamma_s \cdot dW_s \right)$ is a uniformly integrable martingale. Let $c \ge 0$ be an upper bound for the bounded process $|2\beta + |\gamma|^2|$. Then
\begin{align*}
\Ebb \left[ (\Gamma_t)^2 \right] &\le \Ebb \left[ \exp \left\{ \int_0^T |2\beta_s + |\gamma_s|^2|ds \right\} \Ecal \left( \int 2\gamma_s \cdot dW_s \right)_t \right]\\
 &\le e^{cT} \Ebb \left[ \Ecal \left( \int 2\gamma_s \cdot dW_s \right)_t \right]\\
 &= e^{cT}.
\end{align*}
So $\lVert \Gamma \rVert ^2 \le Te^{cT}$. This proves the first part of the lemma.

Now $d\Gamma_t = \Gamma_t (\beta_t dt + \gamma_t \cdot dW_t)$ and $\Gamma_0 = 1$ so
\begin{equation*}
\Gamma_t = 1 + \int_0^t \Gamma_s \beta_s ds + \int_0^t \Gamma_s \gamma_s \cdot dW_s.
\end{equation*}
Observe that, since $\Gamma \in \Hcal^{2,1}_T$, we have that $\Gamma \beta \in \Hcal^{2,1}_T$ and $\Gamma \gamma \in \Hcal^{2,n}_T$ by boundedness. The proof of the second part of this lemma is essentially identical to the method of proving Lemma \ref{supLcal}.
\end{proof}

The next result shows that we can find an {\em explicit formula} for the solutions of LBSDEs that have adjoint processes. It is a slightly modified result of El Karoui (1997):

\begin{prop}\label{LBSDE}
Consider the LBSDE \eqref{genLBSDE} where $\beta$, $\gamma$ are bounded, $\phi \in \Hcal^{2,d}_T$ and $\xi \in \Lcal^{2,d}_T$. Let $\Gamma$ be the adjoint process of the LBSDE. Then the LBSDE has a unique continuous square-integrable solution $(Y,Z)$ where $Y$ is given explicitly by
\begin{equation}\label{LBSDEsoln}
Y_t = \frac{1}{\Gamma_t} \Ebb \left[ \xi \Gamma_T + \displaystyle\int_t^T \Gamma_s \phi_s ds | \Fcal _t \right].
\end{equation}
\end{prop}

\begin{proof}
Notice that with the conditions given, the LBSDE is a BSDE with standard data.  So by Theorem \ref{ExistUniq}, it has a unique continuous square-integrable solution $(Y,Z)$. Define
\begin{equation}
M_t = \Gamma_t Y_t + \displaystyle\int_0^t \Gamma_s \phi_s ds.
\end{equation}
By It\={o}'s formula,
\begin{align*}
dM_t &= \Gamma_t dY_t + Y_t d\Gamma_t + d\langle \Gamma, Y \rangle_t + \Gamma_t \phi_t dt\\
 &= \Gamma_t (-(\phi_t + Y_t\beta_t + Z_t^*\gamma_t)dt + Z_t^*dW_t) + Y_t \Gamma_t (\beta_t dt + \gamma_t \cdot dW_t)\\
 &\quad + \Gamma_t Z_t^*\gamma_t dt + \Gamma_t \phi_t dt.
\end{align*}
All the drift terms cancel and we are left with
\begin{equation}
M_t = Y_0 + \displaystyle\int_0^t \Gamma_s Y_s \gamma_s \cdot dW_s + \displaystyle\int_0^t \Gamma_s Z_s^* dW_s
\end{equation}
which is a local martingale.

Now by Lemmas \ref{supLcal} and \ref{Gammareg}, we have that $\sup_{t \le T} |Y_t| \in \Lcal^{2,1}_T,\  \sup_{t \le T} \Gamma_t \in \Lcal^{2,1}_T$. Hence $\sup_{t \le T} |Y_t| \times \sup_{t \le T} \Gamma_t$ is $\Pbb$-integrable by the Cauchy-Schwarz inequality. Since $\gamma$ is bounded, we therefore have that $\Gamma Y \gamma^* \in  \Hcal^{1,n \times d}_T$. Two applications of H\"{o}lder's inequality (see proof of Proposition \ref{Apriori}) also show us that $\Gamma Z \in \Hcal^{1,n \times d}_T$. Therefore by Corollary \ref{unifmart}, $M$ is a uniformly integrable martingale. In particular for $t \in [0,T]$ we have $M_t = \Ebb [M_T | \Fcal _t]$ from which, recalling the definition of $M$, we get
\begin{equation*}
\Gamma_t Y_t = \Ebb \left[ \xi \Gamma_T + \displaystyle\int_t^T \Gamma_s \phi_s ds | \Fcal _t \right].
\end{equation*}
\end{proof}

The following are immediate consequences of the above proposition:

\begin{cor}\label{LBSDEcor}
Let $d=1$.
\begin{enumerate}
 \item If $\xi$ is $\Pbb$ a.s. non-negative and $\phi$ is $d\Pbb \otimes dt$ a.s. non-negative, then $Y$ is $\Pbb$ a.s. non-negative.
 \item Suppose $\xi$ is $\Pbb$ a.s. non-negative and $\phi$ is $d\Pbb \otimes dt$ a.s. non-negative, and additionally $Y_\tau = 0$ $\Pbb$ a.s. for some constant $\tau \in [0,T]$.

Then $\xi = 0$ $\Pbb$ a.s., $\phi \1bb _{[\tau,T]} = 0$ $d\Pbb \otimes dt$ a.s. and $\forall t \in [\tau,T], \ Y_t = 0$ $\Pbb$ a.s..
\end{enumerate}
\end{cor}

\begin{rem}\label{Xzeta}
With $d=1$, it turns out that for {\em any} (not necessarily square-integrable) continuous solution $(X,\zeta)$ of an LBSDE satisfying the conditions of Proposition \ref{LBSDE}, the first item of Corollary \ref{LBSDEcor} will hold merely if there exists a positive $B \in \Lcal^{2,1}_T$ such that $X_t \ge -B \ \forall t \in [0,T]$ $\Pbb$ a.s.. We can see this in the following way: with $\xi$, $\phi$ non-negative, define $M_t = \Gamma_t X_t + \int_0^t \Gamma_s \phi_s ds$ as before. This is still a local martingale, and is bounded from below by the $\Pbb$-integrable (using Lemma \ref{Gammareg} and Cauchy-Schwarz inequality) random variable $-B\sup_{t \le T}\Gamma_t$. Hence $M$ is a supermartingale, and $M_t \ge \Ebb [M_T | \Fcal _t]$ so
\begin{equation*}
\Gamma_t X_t \ge \Ebb \left[ \xi \Gamma_T + \displaystyle\int_t^T \Gamma_s \phi_s ds | \Fcal _t \right] \ge 0.
\end{equation*}
\end{rem}

\begin{defn}
Sometimes it is convenient to define a family of adjoint processes $(\Gamma^s : s \in [0,T])$, where for $t \in [s,T]$,
\begin{equation*}
\Gamma^s_t := \Ecal \left( \int_s^\cdot \beta_u du + \int_s^\cdot \gamma_u \cdot dW_u \right)_t.
\end{equation*}
This means that $d\Gamma^s_t = \Gamma^s_t (\beta_t dt + \gamma_t \cdot dW_t);\quad \Gamma^s_s = 1$.
\end{defn}

\begin{rem}
Notice that the original definition of the adjoint process $\Gamma$ coincides with $\Gamma^0$. Also notice that for any $0 \le s \le t \le T$,
\begin{equation*}
\frac{\Gamma_t}{\Gamma_s} = \Gamma^s_t.
\end{equation*}
This means that the explicit formula for $Y$ in Proposition \ref{LBSDE} can be written as
\begin{equation}
Y_t = \Ebb \left[ \xi \Gamma^t_T + \displaystyle\int_t^T \Gamma^t_s \phi_s ds | \Fcal _t \right].
\end{equation}
\end{rem}

\subsubsection{The comparison theorem}
This is just a corollary of Proposition \ref{LBSDE}, but is, somewhat miraculously, a result that extends to all BSDEs with standard data. It can be interpreted as an analogue of the maximum principle in PDE theory.

\begin{thm}[Peng (1992)]\label{comparison}
Let $d=1$. Let $(f^1,\xi^1)$, $(f^2,\xi^2)$ be two standard data of BSDEs \eqref{genBSDE} with associated continuous square-integrable solutions $(Y^1,Z^1)$, $(Y^2,Z^2)$. Suppose that the following hold:
\begin{enumerate}
 \item $\xi^1 \ge \xi^2 \quad \Pbb$ a.s.,
 \item $\delta_2 f := f^1(\cdot,Y^2,Z^2) - f^2(\cdot,Y^2,Z^2) \ge 0 \quad d\Pbb \otimes dt$ a.s..
\end{enumerate}
Then $Y^1 \ge Y^2 \quad \Pbb$ a.s..

Moreover, this comparison is strict, i.e. on the event $Y^1_t = Y^2_t$, we have $\xi^1 = \xi^2$, $\delta_2 f_s = 0$ and $Y^1_s = Y^2_s$ a.s. for all $s \ge t$.
\end{thm}

\begin{proof}
For simplicity we assume $n=1$. As in Proposition \ref{Apriori}, let $\delta Y = Y^1 - Y^2$ and $\delta Z = Z^1 - Z^2$. We have that
\begin{equation*}
-d\delta Y_t = \left( f^1(t,Y^1_t,Z^1_t) - f^2(t,Y^2_t,Z^2_t) \right)dt - \delta Z_t dW_t,
\end{equation*}
so if we let
\begin{align*}
\Delta_y f^1(t) &=
\begin{cases}
 \frac{f^1(t,Y^1_t,Z^1_t) - f^1(t,Y^2_t,Z^1_t)}{Y^1_t - Y^2_t} &\quad \text{if } Y^1_t - Y^2_t \ne 0,\\
 0 &\quad \text{otherwise},
\end{cases}\\
\Delta_z f^1(t) &=
\begin{cases}
 \frac{f^1(t,Y^2_t,Z^1_t) - f^1(t,Y^2_t,Z^2_t)}{Z^1_t - Z^2_t} &\quad \text{if } Z^1_t - Z^2_t \ne 0,\\
 0 &\quad \text{otherwise},
\end{cases}
\end{align*}
Then $(\delta Y,\delta Z)$ is the solution of the LBSDE
\begin{equation}
\begin{split}
-d\delta Y_t &= \left( \Delta_y f^1(t) \delta Y_t + \Delta_z f^1(t) \delta Z_t + \delta_2 f_t \right)dt - \delta Z_t dW_t,\\
 \delta Y_T &= \xi^1 - \xi^2.
\end{split}
\end{equation}
Observe that, since the driver $f^1$ is uniformly Lipschitz, $\Delta_y f^1$ and $\Delta_z f^1$ are bounded processes. This LBSDE therefore satisfies the conditions of Proposition \ref{LBSDE}. The theorem then follows from Corollary \ref{LBSDEcor}.

The argument for $n > 1$ is similar, but some additional care has to be taken in defining $\Delta_z f^1(t)$.
\end{proof}

If we set $(f^2,\xi^2)$ to zero, we get a sufficient condition for non-negativity:
\begin{cor}
Let $(f,\xi)$ be the standard data of a BSDE \eqref{genBSDE} with continuous square-integrable solution $(Y,Z)$, and suppose that $\xi \ge 0$ $\Pbb$ a.s. and $f(\cdot,0,0) \ge 0$ $d\Pbb \otimes dt$ a.s.. Then $Y \ge 0$ $\Pbb$ a.s..
\end{cor}

\begin{rem}\label{BSDEbdedbelow}
By Corollary \ref{LBSDEcor} and Remark \ref{Xzeta} we may relax the assumptions of square-integrability of solutions in the comparison theorem, and merely assume that there exists a positive $B \in \Lcal^{2,1}_T$ such that $Y^1_t - Y^2_t \ge -B$ for all $t \in [0,T]$, $\Pbb$ a.s.. We still find that $Y^1 \ge Y^2$ $\Pbb$ a.s..
\end{rem}

Related to this, we get the following extremely useful Corollary:

\begin{cor}\label{solncompare}
Let $d=1$, and let $(Y^1,Z^1)$ be any continuous solution of a BSDE \eqref{genBSDE} with standard data $(f,\xi)$. Suppose there exists a positive $B \in \Lcal^{2,1}_T$ such that $Y^1_t \ge -B \ \forall t \in [0,T]$ $\Pbb$ a.s.. Let $(Y^2,Z^2)$ be the unique continuous square-integrable solution of the BSDE. Then $Y^1 \ge Y^2$ $\Pbb$ a.s..
\end{cor}

\begin{proof}
The two solutions solve the same BSDE, so they satisfy conditions 1 and 2 of the comparison theorem. Additionally using Lemma \ref{supLcal},
\begin{equation*}
Y^1_t-Y^2_t \ge -\left( B+\sup_{t \le T}|Y^2_t| \right) \in \Lcal^{2,1}_T
\end{equation*}
for all $t \in [0,T]$, $\Pbb$ a.s.. The result follows from the previous remark.
\end{proof}

\subsection{Supersolutions}
In this subsection we assume $d=1$.

\begin{defn}
A continuous {\em supersolution} of BSDE \eqref{genBSDE} is a triple $(Y,Z,C) = (Y_t,Z_t,C_t)_{t \in [0,T]}$ of processes such that:
\begin{itemize}
 \item $Y$ is a continuous adapted $\Rbb$-valued process which is bounded below, i.e. there exists positive $B \in \Lcal^{2,1}_T$ with $Y_t \ge -B \ \forall t \in [0,T]$ $\Pbb$ a.s.,
 \item $Z$ is a predictable $\Rbb^{n}$-valued process with $\int_0^T |Z_t|^2ds < \infty$ $\ $ $\Pbb$ a.s.,
 \item $C$ is an increasing continuous adapted $\Rbb$-valued process with $C_0 = 0$,
\end{itemize}
and which satisfies
\begin{equation}
\begin{split}
 -dY_t &= f(t,Y_t,Z_t)dt - Z_t^*dW_t + dC_t;\\
 Y_T &= \xi.
\end{split}
\end{equation}
\end{defn}
Supersolutions are a common concept in mathematical finance, arising naturally in models that incorporate a notion of consumption (hence why the third process is denoted $C$). This link with finance will be covered in more detail in the next section.

\begin{rem}\label{linsupersol}
Suppose $f$ is a linear driver, as in equation \eqref{genLBSDE}, and suppose that $\beta$,$\gamma$ are bounded and that $\phi \in \Hcal^{2,1}_T$. Let $\Gamma$ be the adjoint process of the LBSDE, and let $(Y,Z,C)$ be a supersolution. Defining $M_t = \Gamma_t Y_t + \int_0^t \Gamma_s \phi_s ds$ as in the proof of Proposition \ref{LBSDE}, we have by It\={o}'s formula
\begin{equation*}
dM_t = \Gamma_t Y_t \gamma_t \cdot dW_t + \Gamma_t Z_t \cdot dW_t - \Gamma_t dC_t,
\end{equation*}
so $M$ is a local supermartingale. This leads to the next proposition.
\end{rem}

\begin{prop}\label{supersolncomp}
Let $d=1$. Let the BSDE \eqref{genBSDE} have standard data $(f,\xi)$ and a continuous supersolution $(Y^1,Z^1,C)$, and let $(Y^2,Z^2)$ be its unique continuous square-integrable solution. Then $Y^1 \ge Y^2$ $\Pbb$ a.s..
\end{prop}

\begin{proof}
The proof is analogous to the comparison theorem, so we use the same notation, and assume $n=1$ for simplicity. We have $f^1 = f^2 = f$, so we let $\delta Y$, $\delta Z$, $\Delta_y f(t)$ and $\Delta_z f(t)$ be defined as in the proof of the comparison theorem, so that $(\delta Y,\delta Z,C)$ is a continuous {\em super}solution of a certain LBSDE. Namely, it satisfies
\begin{equation}
\begin{split}
-d\delta Y_t &= \left( \Delta_y f(t) \delta Y_t + \Delta_z f(t) \delta Z_t \right)dt - \delta Z_t dW_t + dC_t,\\
 \delta Y_T &= 0.
\end{split}
\end{equation}
As before, $\Delta_y f(t)$ and $\Delta_z f(t)$ are bounded by the Lipschitz property. Let $\Gamma$ be the adjoint process of the above LBSDE, then we know from Remark \ref{linsupersol} that $\Gamma \delta Y$ is a local supermartingale. However, $\Gamma \delta Y$ is bounded from below by $-\sup_{t \le T}\Gamma_t \left( B + \sup_{t \le T}|Y^2_t| \right)$ which is integrable by the Cauchy-Schwarz inequality, so by Fatou's inequality it is in fact a supermartingale, and in particular for $t \in [0,T]$,
\begin{equation*}
\Gamma_t \delta Y_t \ge \Ebb [ \Gamma_T \delta Y_T | \Fcal _t ] = 0,
\end{equation*}
so $\delta Y \ge 0$ $d\Pbb \otimes dt$ a.s.. The result follows from continuity.
\end{proof}
\newpage

\section{European claims in dynamically complete markets}\label{Eurodyncomp}

In this section we see how BSDEs can help in the solution of one of the most fundamental problems in mathematical finance: the pricing of a European contingent claim.

\subsection{Basic definitions}\label{compdefn}
We keep our probability space $(\Omega,\Fcal,\Pbb)$ from the previous section with fixed time horizon $T > 0$, and define a few new processes:
\begin{itemize}
 \item The {\em short rate} $r = (r_t)_{t \in [0,T]}$, which is $\Rbb$-valued, predictable and bounded,
 \item The vector of {\em stock appreciation rates} $b = (b_t)_{t \in [0,T]}$, which is $\Rbb^n$-valued, predictable and bounded,
 \item The {\em volatility matrix} $\sigma = (\sigma_t)_{t \in [0,T]}$, which is $\Rbb^{n \times n}$-valued, predictable and bounded, and such that $\sigma_t$ is invertible $\Pbb$ a.s. $\forall t \in [0,T]$ with bounded inverse.
\end{itemize}

With these processes defined, our $(n+1)$-asset market is given by the {\em price process} $P = (P^0_t,P^1_t,\ldots,P^n_t)^*_{t \in [0,T]}$ where
\begin{equation}
dP^0_t = P^0_tr_tdt
\end{equation}
is our single {\em locally riskless} asset, representing a bank account or a bond, and 
\begin{equation}
dP^i_t = P^i_t \left( b^i_tdt + \sum_{j=1}^n \sigma^{i,j}_tdW^j_t \right)
\end{equation}
for $i = 1,\ldots,n$ are our {\em risky securities}, for example representing stocks. In addition, we define a predictable and bounded $\Rbb^n$-valued process $\theta = (\theta_t)_{t \in [0,T]}$ known as a {\em risk premium} that satisfies
\begin{equation*}
b_t - r_t \mathbf{1} = \sigma_t \theta_t \quad d\Pbb \otimes dt \ \a.s.
\end{equation*}
where $\mathbf{1} = (1,\ldots,1)^* \in \Rbb ^n$.

Suppose we have an investor whose actions have no affect on the market. We let $V_t$ be his total wealth and $\pi_t = (\pi^1_t,\ldots,\pi^n_t)^*$ the value of his holdings in the $i$th risky asset at time $t \in [0,T]$, so that $\pi^0 := V - \sum_{i=1}^n\pi^i$ is the value of his holdings in the riskless asset. Since he can only decide what to do at time $t$ based on the current information available, we require that the processes $\pi$ and $V$ be predictable.

In the Merton model (1971), the investor also has a {\em consumption rate} $c = (c_t)_{t \in [0,T]}$ which is a scalar non-negative predictable process and represents him putting aside a portion of his wealth, not to be invested further. However in this case we instead usually talk about the {\em total consumption} given by the continuous increasing predictable process $C = \int c_sds$, whose absolute continuity we then relax (so $c$ needn't exist at all).

\begin{defn}
A {\em self-financing trading strategy} is a pair of processes $(V,\pi)$ such that
\begin{equation}\label{selffin}
dV_t = r_tV_tdt + \pi_t \cdot \sigma_t(dW_t + \theta_tdt),
\end{equation}
\begin{equation*}
\int_0^T |\sigma_t^*\pi_t|^2dt < \infty \quad \Pbb \ \a.s..
\end{equation*}
\end{defn}

This SDE is known as the {\em wealth equation}, and is equivalent to
\begin{equation}
V_t = V_0 +  \sum_{i=0}^n\int_0^t \pi^i_t \frac{dP^i_t}{P^i_t},
\end{equation}
and its correct interpretation is that all of the investor's wealth is always invested in some combination of the $n+1$ assets, and that he does not gain or lose wealth in any other manner. We extend these definitions to the Merton model in the following way:

\begin{defn}
A self-financing {\em superstrategy} is a triple of processes $(V,\pi,C)$ such that
\begin{equation}\label{superstrateqn}
dV_t = r_tV_tdt - dC_t + \pi_t \cdot \sigma_t(dW_t + \theta_tdt),
\end{equation}
\begin{equation*}
\int_0^T |\sigma_t^*\pi_t|^2dt < \infty \quad \Pbb \ \a.s..
\end{equation*}
where $C$ is increasing, right-continuous and adapted with $C_0 = 0$. Observe that if $(V,\pi)$ is a self-financing trading strategy, then $(V,\pi,0)$ is a self-financing superstrategy.
\end{defn}

\begin{rem}
Notice that the above two definitions coincide with the definitions of solution and supersolution of a BSDE respectively.
\end{rem}

\begin{defn}
A trading strategy $(V,\pi)$ or superstrategy $(V,\pi,C)$ is {\em admissible} if $V$ is $\Pbb$ a.s. bounded from below. It is {\em 0-admissible} or {\em feasible} if $V \ge 0$ $\Pbb$ a.s..
\end{defn}

\subsection{Hedging claims}
\begin{defn}
A {\em European contingent claim} settled at time $T$ is an $\Fcal _T$-measurable random variable, usually denoted $\xi$.
\end{defn}
Note that this is identical to the definition of $\xi$ in the previous section. Examples of contingent claims include call options and futures contracts.

The questions we would like to answer are as follows: \textbf{Given a non-negative European contingent claim $\xi$, can we find an feasible self-financing strategy $(V,\pi)$ or superstrategy $(V,\pi,C)$ such that we can guarantee that our wealth at time $T$ is $V_T = \xi$? If so, what is the smallest initial wealth $V_0$ needed to carry out such a strategy?} With these questions in mind, we make the following definitions:

\begin{defn}
A {\em hedging strategy} against a European contingent claim $\xi$ is a feasible self-financing trading strategy $(V,\pi)$ such that $V_T = \xi$. Let $\Hfrak (\xi)$ denote the class of all hedging strategies against $\xi$. We call $\xi$ {\em hedgeable} if $\Hfrak (\xi) \ne \emptyset$.

For feasible self-financing superstrategies $(V,\pi,C)$ we similarly define {\em superhedging strategy} against $\xi$, $\Hfrak '(\xi)$ and {\em superhedgeable}.
\end{defn}

\begin{defn}
The {\em fair price} of a hedgeable European contingent claim $\xi$ is
\begin{equation*}
\pfrak (\xi):=\inf \left\{ x \ge 0 : \exists (V,\pi) \in \Hfrak (\xi)\text{ s.t. }V_0 = x \ \a.s. \right\}
\end{equation*}
the lowest initial wealth of a hedging strategy against $\xi$. Similarly we define the {\em upper price} of a superhedgeable European contingent claim $\xi$:
\begin{equation*}
\pfrak '(\xi):=\inf \left\{ x \ge 0 : \exists (V,\pi,C) \in \Hfrak '(\xi)\text{ s.t. }V_0 = x \ \a.s. \right\}.
\end{equation*}
\end{defn}

We can now state the main theorem of this section, which we prove using our results on BSDEs. It shows that all non-negative square-integrable claims are hedgeable, and in this case we call the market {\em dynamically complete}.

\begin{thm}\label{HedgeClaim}
Let $\xi \in \Lcal ^{2,1}_T$ be a non-negative square-integrable European contingent claim. Then there exists some hedging strategy $(X,\pi) \in \Hfrak (\xi)$ achieving the fair price of $\xi$, i.e. $X_0 = \pfrak (\xi)$. Moreover, the upper price $\pfrak '(\xi)$ is equal to the fair price.
\end{thm}

\begin{proof}
This theorem is a simple consequence of BSDE theory. We are looking for a pair $(X,\pi)$ for which
\begin{equation}
\begin{split}
 dX_t &= r_tX_tdt + \pi_t \cdot \sigma_t(dW_t + \theta_tdt),\\
 X_T &= \xi,\\
 X &\ge 0 \quad \Pbb\ \a.s.
\end{split}
\end{equation}
Rearranging this to $-dX_t = (-r_tX_t - (\sigma_t \pi_t)^*\theta_t)dt - (\sigma_t \pi_t)^*dW_t$, we see that this is exactly the form of the LBSDE \eqref{genLBSDE}. Let $(H^s:s \in [0,T])$ be the family of adjoint processes of this LBSDE, i.e.
\begin{equation}\label{deflator}
H^s_t := \Ecal \left( -\int_s^\cdot r_u du - \int_s^\cdot \theta_u \cdot dW_u \right)_t, \quad s\le t.
\end{equation}
Because of our boundedness assumptions, we can apply Proposition \ref{LBSDE} to find the unique square-integrable solution $(X,\sigma\pi)$ of the LBSDE, from which the invertibility of $\sigma$ gives us the process $\pi$. The wealth process $X$ is given by
\begin{equation}\label{minstrat}
X_t = \Ebb \left[ \xi H^t_T | \Fcal _t \right],
\end{equation}
where $\xi$ is non-negative and $H$ is positive so $X$ must be non-negative. So $(X,\pi) \in \Hfrak (\xi)$ (and $(X,\pi,0) \in \Hfrak '(\xi)$).

Now we let $(Y,\rho) \in \Hfrak (\xi)$ be another hedging strategy against $\xi$. Since $Y$ is non-negative, we can apply Corollary \ref{solncompare} to see that $Y \ge X$ $\Pbb$ a.s.. The same holds if we let $(Y,\rho,C) \in \Hfrak '(\xi)$ be a hedging superstrategy against $\xi$. By Proposition \ref{supersolncomp}, $Y \ge X$ $\Pbb$ a.s.. Hence $(X,\pi)$ achieves the fair price and the upper price of $\xi$, given by
\begin{equation}
\pfrak (\xi) = \pfrak '(\xi) = \Ebb \left[ \xi H^0_T \right].
\end{equation}
\end{proof}

\begin{rem}
In finance the family of adjoint processes $(H^s_t)$ given in \eqref{deflator}  is usually referred to as the {\em deflator}.
\end{rem}

\begin{rem}[Equivalent martingale measure]\label{EMM}
Let $H=H^0$ be the deflator started at $0$, and consider the process
\begin{equation*}
e^{\int_0^\cdot r_sds}H = \Ecal \left( - \int_0^\cdot \theta_u \cdot dW_u \right).
\end{equation*}
Since this is in $\Hcal^{2,1}_T$ (by Lemma \ref{Gammareg} and boundedness of $r$), it is a positive uniformly integrable martingale with respect to $\Pbb$ and so we can define a new probability measure $\Qbb$ by the Radon-Nikodym derivative
\begin{equation}
\frac{d\Qbb}{d\Pbb} = e^{\int_0^T r_sds}H_T.
\end{equation}
Under this new measure, the minimal hedging strategy $(X,\pi)$ in equation \eqref{minstrat} satisfies
\begin{align*}
\Ebb ^\Qbb\left[e^{-\int_0^T r_sds}\xi | \Fcal _t \right] &= \frac{\Ebb ^\Pbb\left[H_T\xi | \Fcal _t \right]}{\Ebb ^\Pbb\left[e^{\int_0^T r_sds}H_T | \Fcal _t \right]}\\
 &= \frac{\Ebb ^\Pbb\left[H_T\xi | \Fcal _t \right]}{e^{\int_0^t r_sds}H_t}\\
 &= e^{-\int_0^t r_sds}\Ebb ^\Pbb\left[ \xi H^t_T | \Fcal _t \right]\\
 &= e^{-\int_0^t r_sds}X_t
\end{align*}
by Bayes' rule. This shows that the discounted wealth process $e^{-\int_0^\cdot r_sds}X$ is a $\Qbb$-martingale for {\em any} positive square-integrable claim. The measure $\Qbb$ is called the {\em equivalent martingale measure} or the {\em risk-neutral measure}.
\end{rem}

\begin{rem}
If we allow the definitions of hedging and superhedging strategies to include admissible strategies, Theorem \ref{HedgeClaim} still holds by essentially the same proof.
\end{rem}

We give an example to illustrate what happens when we relax the condition of feasibility (or more generally admissibility) of hedging strategies in Theorem \ref{HedgeClaim}:

\begin{exmp}[From El Karoui, Peng and Quenez (1997)]
Recall Example \ref{Dudley} in the case $d=1$. We can construct an $\Rbb$-valued stochastic integral $\int_0^T \psi_s \cdot dW_s = 1$ such that $\int_0^T \lVert \psi_s \rVert ^2 ds < \infty$, $\Pbb$ a.s.. Construct the pair $(Y,\phi)$ by
\begin{equation}
\begin{split}
 Y_t &= H_t^{-1} \int_0^t \psi_s \cdot dW_s,\\
 \phi_t &= (\sigma_t^*)^{-1}\left( H_t^{-1} \psi_t + Y_t \theta_t \right).
\end{split}
\end{equation}
By an application of It\={o}'s lemma we can show that $(Y,\phi)$ is a self-financing strategy satisfying the LBSDE
\begin{equation*}
 dY_t = r_tY_tdt + \phi_t \cdot \sigma_t(dW_t + \theta_tdt)
\end{equation*}
such that $Y_0 = 0$ and $Y_T = H_T^{-1}$. It is what is known as an {\em arbitrage opportunity}. Now note that the pair $(H^{-1}, H^{-1}(\sigma^*)^{-1}\theta)$ satisfies the same LBSDE, but with $H^{-1}_0 = 1$. So we define
\begin{equation*}
(X^0,\pi^0) = (H^{-1} - Y, H^{-1}(\sigma^*)^{-1}\theta - \phi),
\end{equation*}
which is by linearity a solution of the LBSDE for which $X^0_0 = 1$, $X^0_T = 0$.

Suppose we relax admissibility for hedging strategies, and allow any trading strategy $(V,\rho)$ satisfying the LBSDE and the terminal condition $V_T = \xi$ to be a hedging strategy. Then for any $\lambda \in \Rbb$ and any $(X,\pi) \in \Hfrak (\xi)$, by linearity we have $(X + \lambda X^0,\pi + \lambda \pi^0) \in \Hfrak (\xi)$, which has initial wealth $X_0 + \lambda$. In this case the fair price is not well-defined.
\end{exmp}
\newpage

\section{Concave BSDEs and applications}\label{concave}

We assume $d=1$ in this section. Recall that in Proposition \ref{LBSDE} we derived an explicit solution for an LBSDE, under some regularity conditions. The purpose of this section is to extend the class of standard data for which we can find an explicit solution. We do this by remarking that a concave (resp. convex) function can be expressed as the infimum (resp. supremum) of a collection of linear functions, by means of the Legendre transform. We then give an application of this theory to more general investment problems.

Before we begin we would like to formalize the notion of supremum and infimum of a family of stochastic processes in a way more suited to measure theory, and to this end we follow Dellacherie (1977):

\begin{defn}[Essential supremum and infimum of processes]
Let $U$, $V$ be two processes defined up to time $T$. We say that $U$ {\em minorises} $V$ if
\begin{equation}
\left\{ \omega \in \Omega: \exists t \in [0,T]:\ U_t(\omega)>V_t(\omega) \right\}
\end{equation}
is a $\Pbb$-null set, i.e. $U \le V$ a.s.. Now let $\{ U^\alpha :\alpha \in I \}$ be a family of processes defined up to time $T$, with some indexing set $I$. We say that $U = \ess \inf_\alpha U^\alpha$ if:
\begin{enumerate}
 \item $U$ minorises $U^\alpha$ for every $\alpha \in I$, and
 \item If another process $V$ minorises $U^\alpha$ for every $\alpha \in I$, then $V$ minorises $U$.
\end{enumerate}
We likewise define {\em majorise} and $\ess \sup_\alpha V^\alpha$.
\end{defn}

Recall that the essential supremum and essential infimum of a family of random variables is defined similarly. We now state a result about essential infima without proof:

\begin{lem}[Dellacherie (1977)]
Let $\{ U^\alpha :\alpha \in I \}$ be a family of c\`adl\`ag processes defined up to time $T$, with indexing set $I$. Then $U = \ess \inf_\alpha U^\alpha$ exists and there exists a sequence $(\alpha_n) \in I$ such that $U = \inf_nU^{\alpha_n}$.
\end{lem}

\subsection{Extrema of standard data}
Let $( (f^\alpha,\xi^\alpha):\alpha \in I )$ be a family of data of BSDEs \eqref{genBSDE} with indexing set $I$. What happens if we take the essential infimum of this family as the data of our BSDE? We would like to be able to control solutions of this new BSDE using solutions of our original family of BSDEs. The following proposition is one of a few ways of doing this:

\begin{prop}\label{essinfs}
Let $\left( (f^\alpha,\xi^\alpha):\alpha \in I \right)$ be a family of standard data of BSDEs \eqref{genBSDE} with indexing set $I$ and continuous square-integrable solutions $((Y^\alpha,Z^\alpha):\alpha \in I)$. Let $(f,\xi)$ be another standard data with continuous square-integrable solution $(Y,Z)$. Suppose there exists $\bar\alpha \in I$ such that
\begin{equation}
\begin{split}
 f(\cdot,Y,Z) &= \ess\inf_\alpha f^\alpha (\cdot,Y,Z) = f^{\bar\alpha}(\cdot,Y,Z) \quad d\Pbb \otimes dt\ \a.s.,\\
 \xi &= \ess\inf_\alpha \xi^\alpha = \xi^{\bar\alpha} \quad \Pbb\ \a.s..
\end{split}
\end{equation}
Then the processes $Y$ and $Y^\alpha$ satisfy
\begin{equation}
Y = \ess\inf_\alpha Y^\alpha = Y^{\bar\alpha} \quad \Pbb\ \a.s..
\end{equation}
\end{prop}

\begin{proof}[Proof from Quenez (1997)]
Notice that $(f^\alpha,\xi^\alpha)$ and $(f,\xi)$ are data that satisfy the conditions of the comparison theorem (Theorem \ref{comparison}). So $Y_t \le Y^\alpha _t$ $\forall t \in [0,T]$ $\Pbb$ a.s. for each $\alpha$ and hence $Y \le \ess\inf_\alpha Y^\alpha$ $\Pbb$ a.s. by the definition of essential infimum.

Now we see that $(Y,Z)$ is a continuous square-integrable solution to the BSDE \eqref{genBSDE} with data $(f^{\bar\alpha},\xi^{\bar\alpha})$, so by uniqueness, $(Y,Z)=(Y^{\bar\alpha},Z^{\bar\alpha})$ $\Pbb$ a.s.. Hence
\begin{equation*}
\ess\inf_\alpha Y^\alpha \ge Y = Y^{\bar\alpha} \ge \ess\inf_\alpha Y^\alpha \quad \Pbb\ \a.s..
\end{equation*}
\end{proof}

\subsection{Concave drivers}
Recall the definitions of a {\em concave} function and of a {\em convex} function. Also recall that for any $m \in \Nbb$ we have an involution on the set of convex functions $\left\{ f: \Rbb ^m \to \Rbb ^m \cup \{+\infty\} \right\}$ known as the {\em convex conjugate} (or {\em Legendre transform} in the case $m=1$), and given by $f \mapsto f^*$ where
\begin{equation}
f^*(p) = \sup_{x \in \Rbb ^m} \left( p \cdot x - f(x) \right).
\end{equation}
The conjugate function $f^*(p)$ can be interpreted as $-1$ times the value at $x=0$ of the tangent hyperplane to $f(x)$ with gradient $p$. An important property of the convex conjugate is that it is an involution, i.e. self inverse:
\begin{equation}
f(x) = \sup_{p \in \Rbb ^m} \left( x \cdot p - f^*(p) \right).
\end{equation}
This shows that $f$ can be written as the supremum of a family of linear functions.

We proceed with a modification of this theory. Let $f(t,y,z)$ be a standard generator of a BSDE \eqref{genBSDE} and moreover suppose that $f$ is concave (not convex!) in $(y,z)$. Let $\xi \in \Lcal^{2,1}_T$, and let $(Y,Z)$ be the square-integrable solution of the BSDE with standard data $(f,\xi)$. Let $C$ be a Lipschitz constant for $f$, and let $K= [-C,C]^{n+1} \subseteq \Rbb \times \Rbb ^n$.

\begin{defn}
The {\em polar process} $F: \Omega \times[0,T] \times \Rbb \times \Rbb^n \to \Rbb$ associated with $f$ is the convex function given by
\begin{equation}\label{polar}
F(\omega,t,\beta,\gamma) = \sup_{(y,z) \in \Rbb \times \Rbb ^n}(f(\omega,t,y,z) - \beta y - \gamma \cdot z).
\end{equation}
The {\em effective domain} of $F$ is given by
\begin{equation}
\Dcal_F := \left\{ (\omega,t,\beta,\gamma) \in \Omega \times[0,T] \times \Rbb \times \Rbb^n: F(\omega,t,\beta,\gamma) < \infty \right\}.
\end{equation}
For $(\omega,t) \in \Omega \times[0,T]$, denote the $(\omega,t)$-section of $\Dcal_F$ by $\Dcal_F^{(\omega,t)} \subseteq \Rbb \times \Rbb ^n$.
\end{defn}

\begin{rem}\label{conjrel}
Equation \eqref{polar}, along with the involutive property of the convex conjugate, shows us that the conjugacy relation in this case is
\begin{equation}\label{conjug}
f(\omega,t,y,z) = \inf_{(\beta,\gamma) \in \Dcal_F^{(\omega,t)}}(F(\omega,t,\beta,\gamma) + \beta y + \gamma \cdot z).
\end{equation}
Note that we could extend the infimum above to be over any bounded superset of $\Dcal_F^{(\omega,t)}$, because $F$ is infinite outside of $\Dcal_F^{(\omega,t)}$. This equation has the important interpretation that it allows $f$ to be expressed as the infimum of a family of functions that are linear in $(y,z)$.
\end{rem}

\begin{rem}
For all $(\omega,t) \in \Omega \times[0,T]$, we have $\Dcal_F^{(\omega,t)} \subseteq K$. This is because if, for example, $|\beta| > C$, then the Lipschitz property gives
\begin{equation*}
f(\omega,t,y,z) - \beta y - \gamma \cdot z \ge -C|y| + f(\omega,t,0,z) - \beta y - \gamma \cdot z
\end{equation*}
which is unbounded above in $y$, and hence $(\beta,\gamma) \notin \Dcal_F^{(\omega,t)}$ for any $\gamma \in \Rbb^n$.
\end{rem}

\begin{defn}
Let $(\beta,\gamma)$ be predictable processes, known as {\em control parameters}. Let the linear driver $f^{\beta,\gamma}: \Omega \times[0,T] \times \Rbb \times \Rbb^n \to \Rbb$ be given by
\begin{equation}
f^{\beta,\gamma}(\omega,t,y,z) = F(\omega,t,\beta_t,\gamma_t) + \beta_t y + \gamma_t \cdot z.
\end{equation}
Let the set of {\em admissible control parameters} be given by
\begin{equation}
\Acal = \left\{ (\beta,\gamma) \text{ predictable, $K$-valued}: F(\cdot,\beta,\gamma) \in \Hcal^{2,1}_T \right\}.
\end{equation}
\end{defn}

\begin{rem}
It's easy to see that if $(\beta,\gamma) \in \Acal$, then $f^{\beta,\gamma}$ is a standard driver.
\end{rem}

We seek to apply Proposition \ref{essinfs} to the driver $f$ and the family $\{f^{\beta,\gamma}:(\beta,\gamma) \in \Acal \}$, and need a few lemmas to justify this. The first is given without proof:

\begin{lem}[Measurable selection theorem; Kuratowski--Ryll-Nadzewski (1965)]\label{MeasSelect}
Let $(E,\Ecal)$ be a measurable space and let $X$ be a Polish space. Denote the power set of $X$ by $\Pcal X$. Let $F: E \to \Pcal X$ be a point-to-set mapping and assume that
\begin{enumerate}
 \item $F(\omega) \ne \emptyset$ $\forall \omega \in E$,
 \item For every open set $G \subseteq X$,
\begin{equation*}
\{ \omega \in E: F(\omega) \cap G \ne \emptyset \} \in \Ecal.
\end{equation*}
\end{enumerate}
Then there is an $\Ecal$-measurable function $f$, known as a selection function, such that $f(\omega) \in F(\omega)$ $\forall \omega \in E$.
\end{lem}

The details of this lemma are beyond us, but it suffices for us to know that $\Rbb^m$ is a Polish space for any $m \ge 1$, and also that any closed subset of a Polish space (e.g. $K$) is a Polish space.

\begin{lem}\label{infattained}
For any $(\omega,t,y,z)$, the infimum in the conjugacy relation \eqref{conjug} is achieved in $K$ by some pair $(\beta,\gamma)$.
\end{lem}

\begin{proof}[Proof from El Karoui, Peng and Quenez (1997)]
We fix a quadruple $(\omega,t,y,z)$. By the infimum \eqref{conjug}, there exists a sequence $(\beta^n,\gamma^n)_{n \in \Nbb} \in \Dcal_F^{(\omega,t)}$ such that
\begin{equation*}
f(\omega,t,y,z) = \lim_{n \to \infty}(F(\omega,t,\beta^n,\gamma^n) + \beta^n y + \gamma^n \cdot z).
\end{equation*}
Since $\Dcal_F^{(\omega,t)}$ is contained in a compact set, we can assume without loss of generality that $(\beta^n,\gamma^n)_n$ converges to some $(\beta,\gamma) \in K$ by the Bolzano--Weierstrass theorem. $F(\omega,t,\cdot,\cdot)$ is a convex conjugate, so by a well-known result in convex analysis, it is lower semi-continuous. Therefore
\begin{align*}
F(\omega,t,\beta,\gamma) + \beta y + \gamma \cdot z &\le \lim_{n \to \infty}(F(\omega,t,\beta^n,\gamma^n) + \beta^n y + \gamma^n \cdot z)\\
 &= f(\omega,t,y,z)\\
 &= \inf_{(\beta',\gamma') \in K}(F(\omega,t,\beta',\gamma') + \beta' y + \gamma' \cdot z)\\
 &\le F(\omega,t,\beta,\gamma) + \beta y + \gamma \cdot z
\end{align*}
where we have extended the infimum to be over the set $K$ by Remark \ref{conjrel}. So
\begin{equation*}
f(\omega,t,y,z) = F(\omega,t,\beta,\gamma) + \beta y + \gamma \cdot z.
\end{equation*}
\end{proof}

We now prove that the main condition needed to invoke Proposition \ref{essinfs} holds in this case:

\begin{lem}\label{OptimContr}
There exists an optimal control $(\bar\beta,\bar\gamma) \in \Acal$ such that
\begin{equation}
f(\cdot,Y,Z) = f^{\bar\beta,\bar\gamma}(\cdot,Y,Z).
\end{equation}
\end{lem}

\begin{proof}[Proof from El Karoui, Peng and Quenez (1997)]
We aim to construct a point-to-set function from the measurable space $(\Omega \times [0,T],\Pfrak)$ to the Polish space $(K,\Bcal ^{n+1})$ and then use the measurable selection theorem \ref{MeasSelect}. By Lemma \ref{infattained}, given $(\omega,t) \in \Omega \times [0,T]$, the set of all $(\beta,\gamma) \in K$ that minimise \eqref{conjug} in the case that $f(\omega,t,y,z) = f(\omega,t,Y_t(\omega),Z_t(\omega))$ is non-empty. Moreover, the predictability of $f(\cdot,Y,Z)$, $Y$ and $Z$ ensure that the second condition of the measurable selection theorem holds. Hence we can find a predictable pair $(\bar\beta,\bar\gamma)$ taking values in $K$ such that for all $(\omega,t)$,
\begin{equation}
f^{\bar\beta,\bar\gamma}(\omega,t,Y_t(\omega),Z_t(\omega)) = \inf_{(\beta,\gamma) \in \Dcal_F^{(\omega,t)}}(F(\omega,t,\beta,\gamma) + \beta Y_t(\omega) + \gamma \cdot Z_t(\omega))
\end{equation}
and hence
\begin{equation}
f(\cdot,Y,Z) = f^{\bar\beta,\bar\gamma}(\cdot,Y,Z)
\end{equation}
where we have suppressed the dependence on $\omega$. Now since $f(\cdot,Y,Z)$, $Y$ and $Z$ are square-integrable and $(\bar\beta,\bar\gamma)$ is bounded, the equation
\begin{equation*}
F(\cdot,\bar\beta,\bar\gamma) = f(\cdot,Y,Z) - \bar\beta Y - \bar\gamma \cdot Z
\end{equation*}
illustrates that $F(\cdot,\bar\beta,\bar\gamma) \in \Hcal^{2,1}_T$. So $(\bar\beta,\bar\gamma) \in \Acal$.
\end{proof}

Now for each $(\beta,\gamma) \in \Acal$, we denote the square-integrable solution of the associated LBSDE with data $(f^{\beta,\gamma},\xi)$ by $(Y^{\beta,\gamma},Z^{\beta,\gamma})$. We now state our main theorem:

\begin{thm}\label{concaveBSDE}
Let $f$ be a concave standard driver and $\{f^{\beta,\gamma}:(\beta,\gamma) \in \Acal \}$ the associated linear standard drivers satisfying
\begin{equation*}
f = \ess\inf_{(\beta,\gamma) \in \Acal} f^{\beta,\gamma} \quad d\Pbb \otimes dt \ \a.s..
\end{equation*}
Then
\begin{equation*}
Y = \ess\inf_{(\beta,\gamma) \in \Acal} Y^{\beta,\gamma}\quad \Pbb \ \a.s..
\end{equation*}
\end{thm}

\begin{proof}[Proof from El Karoui, Peng and Quenez (1997)]
From Lemma \ref{OptimContr} we have that
\begin{equation*}
f(\cdot,Y,Z) = f^{\bar\beta,\bar\gamma}(\cdot,Y,Z) = \ess\inf_{(\beta,\gamma) \in \Acal} f^{\beta,\gamma}(\cdot,Y,Z) \quad d\Pbb \otimes dt \ \a.s..
\end{equation*}
so the result follows directly from Proposition \ref{essinfs}.
\end{proof}

Now let's use what we know about LBSDEs. Take $(\beta,\gamma) \in \Acal$, so that $(\beta,\gamma)$ is bounded, predictable, and $F(\cdot,\beta,\gamma) \in \Hcal^{2,1}_T$. Let $(\Gamma ^{\beta,\gamma}_{s,t})_{t \in [s,T]}$ denote the adjoint process associated with the LBSDE with standard data $(f^{\beta,\gamma},\xi)$ and solution $(Y^{\beta,\gamma},Z^{\beta,\gamma})$. Then $Y^{\beta,\gamma}$ is given by
\begin{equation*}
Y^{\beta,\gamma}_t = \Ebb \left[ \xi \Gamma ^{\beta,\gamma}_{t,T} + \displaystyle\int_t^T \Gamma ^{\beta,\gamma}_{t,s} F(s,\beta_s,\gamma_s) ds | \Fcal _t \right]
\end{equation*}
and so we have an explicit expression for $(Y_t)_{t \le T}$:
\begin{equation}
Y_t = \ess\inf_{(\beta,\gamma) \in \Acal} \Ebb \left[ \xi \Gamma ^{\beta,\gamma}_{t,T} + \displaystyle\int_t^T \Gamma ^{\beta,\gamma}_{t,s} F(s,\beta_s,\gamma_s) ds | \Fcal _t \right].
\end{equation}

\subsection{Application: European claims in convex markets}
Recall the self-financing condition of a trading strategy in our dynamically complete market of Section \ref{Eurodyncomp}, represented by the following BSDE:
\begin{align*}
dV_t &= r_tV_tdt + \pi_t \cdot \sigma_t(dW_t + \theta_tdt)\\
 &= (r_t V_t + \pi_t \cdot \sigma_t \theta_t)dt + \pi_t \cdot \sigma_t dW_t.
\end{align*}
Sometimes the market might be a bit more interesting. Note that the pair $(V, \sigma^* \pi)$ in the above wealth equation corresponds to the solution $(Y,Z)$ of a BSDE, so we now consider a general wealth equation
\begin{equation}\label{genwealth}
-dY_t = b(t,Y_t,Z_t)dt - Z_t^* dW_t,
\end{equation}
where $b: \Omega \times [0,T] \times \Rbb \times \Rbb^n \to \Rbb$ is a standard driver, and the volatility matrix $\sigma$ satisfies our previous assumptions, i.e. $\Rbb^{n \times n}$-valued, predictable and bounded, and such that $\sigma_t$ is invertible $\Pbb$ a.s. $\forall t \in [0,T]$ with bounded inverse. In Section \ref{Eurodyncomp} we took
\begin{equation*}
b(t,y,z) = -r_ty-z^*\theta_t.
\end{equation*}
Let $C$ be a Lipschitz constant for $b$ and let $K= [-C,C]^{n+1} \subseteq \Rbb \times \Rbb^n$. We apply our previous results by assuming that $b(t,y,z)$ is convex in $(y,z)$. If $(V,\sigma^* \pi)$ is a solution of \eqref{genwealth}, then $(-V,-\sigma^* \pi)$ is a solution of the BSDE
\begin{equation}\label{negb}
-dY_t = (-b(t,-Y_t,-Z_t))dt - Z_t^* dW_t,
\end{equation}
and notice that $-b(t,-y,-z)$ is a standard driver with Lipschitz constant $C$, and is concave in $(y,z)$. Let $B$ be the convex polar process associated with $-b(t,-y,-z)$, given by
\begin{align*}
B(\omega,t,\beta,\gamma) &= \sup_{(y,z) \in \Rbb \times \Rbb ^n}(-b(\omega,t,-y,-z) - \beta y - \gamma \cdot z)\\
 &= \sup_{(y,z) \in \Rbb \times \Rbb ^n}( \beta y + \gamma \cdot z - b(\omega,t,y,z)),
\end{align*}
so that the conjugacy relation is
\begin{align}\label{conjugB}
-b(\omega,t,-y,-z) &= \inf_{(\beta,\gamma) \in B}(B(\omega,t,\beta,\gamma) + \beta y + \gamma \cdot z),\\
\Rightarrow \qquad b(\omega,t,y,z) &= \sup_{(\beta,\gamma) \in B}(\beta y + \gamma \cdot z - B(\omega,t,\beta,\gamma)).
\end{align}
Note that this means that $(b,B)$ are a pair of convex conjugates. Let the set of admissible control parameters be
\begin{equation*}
\Acal = \left\{ (\beta,\gamma) \text{ predictable, $K$-valued}: B(\cdot,\beta,\gamma) \in \Hcal^{2,1}_T \right\}.
\end{equation*}
For $(\beta,\gamma) \in \Acal$, let
\begin{equation}
b^{\beta,\gamma}(\omega,t,y,z) :=\beta_t y + \gamma_t \cdot z - B(\omega,t,\beta_t,\gamma_t).
\end{equation}
be a linear standard driver. We can now apply Theorem \ref{concaveBSDE}:

\begin{prop}\label{convhedge}
Let $\xi \in \Lcal^{2,1}_T$. Let $(X,\sigma^* \pi)$ be the unique square-integrable solution of BSDE \eqref{genwealth} with convex standard parameters $(b,\xi)$. Let $(X^{\beta,\gamma},\sigma^*\pi^{\beta,\gamma})$ be the square-integrable solution of the LBSDE with standard data $(b^{\beta,\gamma},\xi)$. Then $\Pbb$ a.s.,
\begin{equation}
X = \ess\sup_{(\beta,\gamma) \in \Acal} X^{\beta,\gamma}.
\end{equation}
\end{prop}

\begin{proof}
Recall that $(-X,-\sigma^* \pi)$ is the unique square-integrable solution to the BSDE \eqref{negb}. The pair $(-X^{\beta,\gamma},-\sigma^*\pi^{\beta,\gamma})$ satisfies the LBSDE
\begin{equation*}
 -dY_t = -b^{\beta,\gamma}(t,-Y_t,-Z_t)dt - Z_t^*dW_t.
\end{equation*}
Now
\begin{equation*}
-b^{\beta,\gamma}(\omega,t,-y,-z) =\beta_t y + \gamma_t \cdot z + B(\omega,t,\beta_t,\gamma_t)
\end{equation*}
so from equation \eqref{conjugB},
\begin{equation*}
-b(\cdot,-Y,-Z) = \ess\inf_{(\beta,\gamma) \in \Acal} (-b^{\beta,\gamma}(\cdot,-Y,-Z)) \quad d\Pbb \otimes dt \ \a.s..
\end{equation*}
By Theorem \ref{concaveBSDE},
\begin{equation*}
-X = \ess\inf_{(\beta,\gamma) \in \Acal} (-X^{\beta,\gamma})\quad \Pbb \ \a.s.
\end{equation*}
and the result follows.
\end{proof}

\begin{rem}
We have that $(X^{\beta,\gamma},\sigma^*\pi^{\beta,\gamma})$ satisfies the LBSDE
\begin{equation}
\begin{split}
 -dY_t &= (\beta_t Y_t + \gamma_t \cdot Z_t - B(t,\beta_t,\gamma_t))dt - Z_t^*dW_t;\\
 Y_T &= \xi.
\end{split}
\end{equation}
If we compare this to equation \eqref{superstrateqn}, we see that $(X^{\beta,\gamma},\pi^{\beta,\gamma})$ is a hedging strategy against $\xi$ in a fictitious market with bounded short rate $-\beta$, bounded risk premium $-\gamma$ and (not necessarily non-negative) instantaneous cost process $- B(\cdot,\beta,\gamma)$. Moreover, the price of the claim $\xi$ is equal to its price in an optimal such fictitious market.
\end{rem}

We now give a concrete example, demonstrating the applicability of this result.

\begin{exmp}[Markets with higher interest rate for borrowing]
Let us suppose that holding a negative quantity of the riskless asset incurs a higher interest rate $R > r$, where $r$ is the original short rate. We assume the process $R$ to be predictable and bounded. The self-financing condition in this case is thus
\begin{equation}\label{plusinterest}
dV_t = r_t V_t dt + \pi_t \cdot \sigma_t (\theta_t dt + dW_t) - (R_t - r_t) (\pi^0_t)^- dt
\end{equation}
where $(\pi^0)^- $ is the negative part of the holdings of the riskless asset $\pi^0$ (and is a non-negative process). Let's eliminate $\pi^0$ and rearrange into standard SDE form:
\begin{equation}
-dV_t = -\left( r_t V_t+ \pi_t \cdot \sigma_t \theta_t  - (R_t - r_t) (V_t - \pi_t \cdot \mathbf{1})^- \right) dt - \pi_t \cdot \sigma_t dW_t
\end{equation}
where $\mathbf{1} = (1,\ldots,1)^* \in \Rbb ^n$. Our driver in this case is given by
\begin{equation}
b(t,y,z) = - r_t y - \theta_t \cdot z + (R_t - r_t) (y - (\sigma_t^{-1}\mathbf{1}) \cdot z)^-
\end{equation}
which, with respect to $(y,z)$, is the sum of a linear function and a convex function, so is convex. The processes $r$, $\theta$, $R$ and $\sigma^{-1}$ are all bounded so $b$ is a standard driver. Let $\xi \in \Lcal^{2,1}_T$, and let $(X,\sigma^*\pi)$ be the square-integrable solution of the BSDE \eqref{genBSDE} with standard data $(b,\xi)$. The convex conjugate $B$ associated with $b$ is given by
\begin{align*}
B(\omega,t,\beta,\gamma) &=  \sup_{(y,z) \in \Rbb \times \Rbb ^n}( \beta y + \gamma \cdot z - b(\omega,t,y,z))\\
 &= \sup_{(y,z) \in \Rbb \times \Rbb ^n}( (\beta + r_t(\omega)) y + (\gamma + \theta_t(\omega)) \cdot z\\
 &\qquad \qquad \qquad - (R_t(\omega) - r_t(\omega)) (y - (\sigma_t(\omega)^{-1}\mathbf{1}) \cdot z)^-).
\end{align*}
The expression above takes a finite maximum over $(y,z)$ if and only if
\begin{align*}
-(R_t(\omega) - r_t(\omega)) &\le \beta + r_t(\omega) \le 0,\\
\gamma + \theta_t(\omega) &= -(\beta + r_t(\omega))\sigma_t(\omega)^{-1}\mathbf{1},
\end{align*}
and in this case
\begin{align*}
B(\omega,t,\beta,\gamma) &= \sup_{(y,z) \in \Rbb \times \Rbb ^n}( (\beta + r_t(\omega)) (y -(\sigma_t(\omega)^{-1}\mathbf{1}) \cdot z)\\
 &\qquad \qquad \qquad - (R_t(\omega) - r_t(\omega)) (y - (\sigma_t(\omega)^{-1}\mathbf{1}) \cdot z)^-).
\end{align*}
So $B$ is given by
\begin{equation}
B(\omega,t,\beta,\gamma) =
\begin{cases}
0 \quad \text{if } \beta \in [-R_t(\omega),r_t(\omega)] \text{, } \gamma = -\theta_t(\omega) -(\beta + r_t(\omega)) \sigma_t(\omega)^{-1}\mathbf{1};\\
\infty \quad \text{otherwise}.
\end{cases}
\end{equation}
Thus in this case, the set of admissible control parameters is
\begin{equation*}
\Acal = \left\{ (\beta,\gamma) \text{ predictable}: -R \le \beta \le -r, \gamma = -\theta -(\beta + r) \sigma^{-1}\mathbf{1} \right\}.
\end{equation*}
Notice that we no longer require that $(\beta,\gamma)$ take values in some bounded set $K$. This is because the conditions given already bound $(\beta,\gamma)$, and the Lipschitz constant of $b$ can be made arbitrarily large.

For a predictable process $\beta$ such that $-R \le \beta \le -r$, we hence define $(X^\beta,\sigma^*\pi^\beta)$ to be the solution of the LBSDE with standard data given by
\begin{equation}
\begin{split}
 -dY_t &= (\beta_t Y_t - (\theta_t + (\beta_t + r_t) \sigma_t^{-1}\mathbf{1}) \cdot Z_t)dt - Z_t^*dW_t;\\
 Y_T &= \xi.
\end{split}
\end{equation}
Then by Proposition \ref{convhedge}, it follows that $\Pbb$ a.s.,
\begin{equation}
X = \ess\sup_{-R \le \beta \le -r} X^\beta.
\end{equation}
In this case, for all admissible $(\beta,\gamma)$ we have that $B = 0$ so if $\xi$ is non-negative, so $X^\beta$ is non-negative by Corollary \ref{LBSDEcor}, and so $X$ is $\Pbb$ a.s. non-negative. $(X,\pi)$ is therefore a feasible square-integrable hedging strategy against $\xi$. By Corollary \ref{solncompare} and Proposition \ref{supersolncomp}, it follows that the fair price and the upper price of $\xi$ are given by
\begin{equation}
\pfrak (\xi) = \pfrak '(\xi) = \sup_{-R \le \beta \le -r} \Ebb \left[ \Gamma^\beta_T \xi \right],
\end{equation}
where
\begin{equation*}
\Gamma^\beta_T = \Ecal \left( \int \beta_s ds - \int (\theta_s + (\beta_s + r_s) \sigma_s^{-1}\mathbf{1}) \cdot dW_s \right)_T.
\end{equation*}
Finally, notice that if we take $\beta' = -\beta$ so that $r \le \beta' \le R$, then the strategy $(X^\beta,\pi^\beta)$ satisfies
\begin{equation}
\begin{split}
 dX^\beta_t &= (\beta'_t X^\beta_t + \pi^\beta_t \cdot \sigma_t \theta_t - (\beta'_t - r_t ) \pi^\beta_t \cdot \mathbf{1})dt + \pi^\beta_t \cdot \sigma_t dW_t \\
 &= r_t X^\beta_tdt + \pi^\beta_t \cdot \sigma_t (\theta_t dt + dW_t ) + (\beta'_t - r_t ) (\pi^\beta)^0_tdt
\end{split}
\end{equation}
which is a strategy against $\xi$ in a fictitious dynamically complete market with a short rate $r$ and risk premium $\theta$, and an instantaneous cost process that depends linearly on $(\pi^\beta)^0$, the quantity of wealth held in the riskless asset. Notice the similarities between this SDE and equation \eqref{plusinterest}, as expected.
\end{exmp}
\newpage

\section{Utility theory in incomplete markets}\label{utility}

In this section we consider the problem of finding a portfolio process that maximises the expected utility of our time-$T$ wealth $X_T$ in a potentially incomplete market. That is, given a {\em utility function} $U$ taking values in $\Rbb$, an initial value $x \in \Rbb$ and a set of {\em admissible portfolio processes} $\Acal$, our objective is to maximise
\begin{equation*}
\Ebb \left[ U\left( X^{(\pi)}_T \right) \right]
\end{equation*}
where $(X^{(\pi)},\pi)$ is a trading strategy such that $\pi \in \Acal$ and $X^{(\pi)} = x$. The utility function $U$ represents how much ``value'' our investor assigns to a given quantity of wealth, and is frequently assumed to have various nice properties such as monotonicity and concavity. Examples of utility functions include:
\begin{enumerate}
 \item {\em Logarithmic utility}: $U: (0,\infty) \to \Rbb$ such that
\begin{equation*}
U(x) = \log x.
\end{equation*}
 \item{\em Power utility}: $U_\gamma: [0,\infty) \to \Rbb$ with $\gamma \in (0,1)$ such that
\begin{equation*}
U_\gamma (x) = \frac{1}{\gamma} x^\gamma.
\end{equation*}
 \item{\em Exponential utility}: $U_\alpha: \Rbb \to \Rbb$ with $\alpha \in (0,\infty)$ such that
\begin{equation*}
U_\alpha (x) = - e^{-\alpha x}.
\end{equation*}
\end{enumerate}
Note that all of these functions satisfy both monotonicity and concavity.

\subsection{Preliminaries}

\begin{defn}
Let $A$ be a closed subset of $\Rbb^m$ and let $a \in \Rbb^m$. Let the {\em distance} between $a$ and $A$ be
\begin{equation}
\dist _A (a) = \min_{b \in A} |a-b|,
\end{equation}
the minimum distance between $a$ and an element of $A$. Obviously if $a \in A$ then $\dist _A (a) = 0$. Let the set
\begin{equation}
\Pi _A (a) = \left\{ b \in A: |a-b| = \dist _A (a) \right\}
\end{equation}
be the set of elements of $A$ attaining the minimum distance from $a$. The set $A$ is closed, so $\Pi _A (a)$ is non-empty, and may contain more than one element.
\end{defn}

We give without proof a measurability result relating to the distance function defined above. A proof, based on joint continuity, can be found in Hu, Imkeller and M\"{u}ller (2005).

\begin{lem}\label{distpred}
Let $(a_t)_{t \in [0,T]}$ and $(\sigma_t)_{t \in [0,T]}$ be $\Rbb^n$- and $\Rbb^{d \times n}$-valued predictable processes respectively. Let $\tilde C$ be a closed subset of $\Rbb^d$. Let $C_t = \sigma_t^*\tilde C$. Then the process
\begin{equation*}
d = (\dist (a_t,C_t))_{t \in [0,T]}
\end{equation*}
is predictable.
\end{lem}

Our setup in this section differs from Section \ref{Eurodyncomp} in a number of ways:
\begin{enumerate}
 \item The volatility matrix $\sigma$ is now generalised to a full-rank predictable $\Rbb^{d \times n}$-valued process, where $d \le n$. The number $d$ is the number of risky assets in the market, and $n$ is the dimension of the Brownian motion $W$, which can be interpreted as the number of ``degrees of freedom'' in the market. Additionally, $\sigma \sigma^*$ is assumed to be {\em uniformly elliptic}, which means that there exist $\epsilon , K > 0$ such that $\Pbb$ a.s.,
\begin{equation}
\epsilon |x| \le |\sigma_t^* x| \le K |x| \quad \forall x \in \Rbb^d,\ \forall t \in [0,T].
\end{equation}
 \item Since $\sigma$ may no longer be invertible, we now define the $\Rbb^n$-valued risk premium $\theta$ by
\begin{equation}
\theta = \sigma^* (\sigma \sigma^*)^{-1} b \quad \Pbb \ \a.s.
\end{equation}
where $b$ is the predictable, bounded and $\Rbb^d$-valued vector of stock appreciation rates. This agrees with the original definition if $d=n$ and $\sigma$ is invertible. Since $b$ is bounded and $\sigma \sigma^*$ is uniformly elliptic, we have that $\theta$ is bounded.
 \item This is a constrained optimisation problem: We require that our portfolio processes only take values inside some (non-empty(!)) closed subset. Note that this subset is not assumed to be convex.
\end{enumerate}

We keep the same definition of a self-financing trading strategy, i.e. $(X,\pi)$ satisfying
\begin{equation*}
dX_t = r_t X_t dt + \pi_t \cdot \sigma_t(dW_t + \theta_tdt),
\end{equation*}
\begin{equation*}
\int_0^T |\sigma_t^*\pi_t|^2dt < \infty \quad \Pbb \ \a.s..
\end{equation*}

\begin{rem}
If $d < n$ then the market is called {\em incomplete}, since there exist square-integrable claims that cannot be hedged by a feasible trading strategy.
\end{rem}

\begin{defn}\label{proportion}
Sometimes it can be more instructive to describe the entries of the portfolio process $\pi$ by the proportion of the total wealth $X$ that they represent. For a trading strategy $(X,\pi)$ such that $X$ is positive, we therefore define the $\Rbb^d$-valued predictable process $\tilde\rho$ by
\begin{equation}
\tilde\rho^i_t =\frac{\pi^i_t}{X_t}
\end{equation}
so that the wealth equation can be written as
\begin{equation}
dX_t = X_t(r_tdt + \tilde\rho_t \cdot \sigma_t(dW_t + \theta_tdt)).
\end{equation}
This is useful because it allows us to write the total wealth as a stochastic exponential:
\begin{equation}
X = X_0 \Ecal \left( \int r_s ds + \int \tilde\rho_s \cdot \sigma_s(dW_s + \theta_sds) \right).
\end{equation}
Finally, we define the $\Rbb^n$-valued predictable process $\rho$ by $\rho = \sigma^* \tilde\rho$, for additional simplicity.
\end{defn}

\subsection{Logarithmic utility}
This is possibly the simplest of the three utility functions given above over which to optimise. We provide a generalisation of a result of Hu, Imkeller and M\"{u}ller (2005), who assumed the short rate $r$ to be zero.

As a constraint, we require that our process $\tilde\rho$ (see Definition \ref{proportion}) take values in some fixed non-empty closed subset $\tilde C \subseteq \Rbb^d$. In this section we will mainly be working with the process $\rho = \sigma^* \tilde\rho$, so for $t \in [0,T]$ we define
\begin{equation}
C_t = \sigma^*_t\tilde C.
\end{equation}
Note that by working with $\tilde\rho$ we are implicitly assuming that our wealth processes $X$ are positive.

\begin{defn}
Our set of admissible processes $\Acal_l$ is defined as the set of processes $\rho \in \Hcal^{2,n}_T$ such that $\rho_t \in C_t$ $d\Pbb \otimes dt$ a.s..
\end{defn}

Given $\rho \in \Acal_l$, denote the wealth process associated with $\rho$ by $X^{(\rho)}$, i.e.
\begin{equation}
X^{(\rho)} = X^{(\rho)}_0 \Ecal \left( \int r_s ds + \int \rho_s \cdot (dW_s + \theta_sds) \right).
\end{equation}
Fixing some $x > 0$, our objective is thus to maximise
\begin{equation*}
\Ebb \left[ \log \left( X^{(\rho)}_T \right) \right]
\end{equation*}
over pairs $(X^{(\rho)},\rho)$ such that $\rho \in \Acal_l$ and $X^{(\rho)}_0 = x$. Let
\begin{equation}
V(x) = \sup_{\substack{\rho \in \Acal_l \\ X^{(\rho)}_0 = x}} \Ebb \left[ \log \left( X^{(\rho)}_T \right) \right]
\end{equation}
be the value we are trying to find. It follows that
\begin{equation}\label{logvalfunc}
V(x) = \log(x) + \sup_{\rho \in \Acal_l} \Ebb \left[ \int_0^T \rho_s \cdot dW_s + \int_0^T \left(r_s + \rho_s \cdot \theta_s - \frac{1}{2}|\rho_s|^2\right)ds \right].
\end{equation}
Here is the rough plan of attack: We seek to define a suitable family of processes $(R^{(\rho)})$ indexed by $\rho \in \Acal_l$ such that
\begin{enumerate}
 \item $R^{(\rho)}_0 = R_0$ is constant,
 \item $R^{(\rho)}_T = \int_0^T \rho_s \cdot dW_s + \int_0^T (r_s + \rho_s \cdot \theta_s - \frac{1}{2}|\rho_s|^2)ds$ for all $\rho \in \Acal_l$ (as in \eqref{logvalfunc}),
 \item $R^{(\rho)}$ is a supermartingale for all $\rho \in \Acal_l$,
 \item $\exists \hat\rho \in \Acal_l$ such that $R^{(\hat\rho)}$ is a martingale.
\end{enumerate}
It will follow that $(X^{(\hat\rho)},\hat\rho)$ attains the maximum $V(x)$. To do this, we introduce the BSDE
\begin{equation}\label{optimBSDE}
\begin{split}
 -dY_t &= f(t,Y_t,Z_t)dt - Z_t^*dW_t;\\
 Y_T &= 0,
\end{split}
\end{equation}
where $f$ is an $\Rbb$-valued driver, to be specified later. This is equivalent to
\begin{equation*}
Y_t = \int_t^T f(s,Y_s,Z_s)ds - \displaystyle\int_t^T Z_s^*dW_s.
\end{equation*}
Assume that the BSDE \eqref{optimBSDE} (which we still haven't actually specified) has a solution $(Y,Z)$. We can hence define
\begin{equation*}
\begin{split}
 R^{(\rho)}_t &= \int_0^t \rho_s \cdot dW_s + \int_0^t \left(r_s -\frac{1}{2}|\rho_s - \theta_s|^2 + \frac{1}{2}|\theta_s|^2\right)ds + Y_t\\
 &= Y_0 + \int_0^t (\rho_s + Z_s) \cdot dW_s + \int_0^t \left(r_s -\frac{1}{2}|\rho_s - \theta_s|^2 + \frac{1}{2}|\theta_s|^2 - f(s,Y_s,Z_s)\right)ds,
\end{split}
\end{equation*}
which can be written as
\begin{equation}
d R^{(\rho)}_t =  (\rho_t + Z_t) \cdot dW_t + \left(r_t -\frac{1}{2}|\rho_t - \theta_t|^2 + \frac{1}{2}|\theta_t|^2 - f(t,Y_t,Z_t)\right)dt
\end{equation}
since the value of $Y_0$ is irrelevant (all we need is that it is independent of $\rho$). Notice that $R^{(\rho)}_T$ satisfies the terminal condition that we want. To get the required supermartingale property for $R^{(\rho)}$, we would like the drift term here to be non-positive, but not strictly negative. Hence we pick
\begin{equation}
f(t,y,z) = r_t + \frac{1}{2}|\theta_t|^2 - \frac{1}{2} \dist (\theta_t,C_t)^2
\end{equation}
which is in fact independent of its last two arguments, and predictable by Lemma \ref{distpred}. Therefore
\begin{equation}
d R^{(\rho)}_t =  (\rho_t + Z_t) \cdot dW_t - \left(\frac{1}{2}|\rho_t - \theta_t|^2 - \frac{1}{2} \dist (\theta_t,C_t)^2 \right)dt.
\end{equation}
This is all well and good, but now we have to show that a solution to our BSDE \eqref{optimBSDE} actually exists.

\begin{lem}
There is a unique continuous square-integrable solution $(Y,Z)$ to \eqref{optimBSDE}.
\end{lem}

\begin{proof}
Since the driver $f$ has no dependence on $(y,z)$, a sufficient condition for us to use our existence-uniqueness theorem is that $f(\cdot)$ is $\Pbb$ a.s. bounded.

The short rate $r$ and risk premium $\theta$ are both $\Pbb$ a.s. bounded, so let $M_1>0$ be an upper bound for $|r|$ and $M_2>0$ an upper bound for $|\theta|$. The constraint set $\tilde C$ is non-empty, so let $c \in \tilde C$. Then we have that $\sigma^*_t c \in C_t$ $\forall t \in [0,T]$. Then by uniform ellipticity, $\Pbb$ a.s. for all $t \in [0,T]$,
\begin{align*}
\dist (\theta_t,C_t) &\le |\theta_t - \sigma^*_t c|\\
 &\le |\theta_t| + |\sigma^*_t c|\\
 &\le M_2 + K|c|
\end{align*}
so
\begin{equation*}
|f(t)| \le M_1 + \frac{1}{2} M_2^2 +\frac{1}{2} (M_2 + K|c|)^2
\end{equation*}
and hence $f(\cdot)$ is $\Pbb$ a.s. bounded. Thus by Theorem \ref{ExistUniq} there is a unique square-integrable solution.
\end{proof}

Let us return to $R^{(\rho)}$, now given by
\begin{equation}
R^{(\rho)}_t = Y_0 + \int_0^t (\rho_s + Z_s) \cdot dW_s + \int_0^t -\left(\frac{1}{2}|\rho_s - \theta_s|^2 - \frac{1}{2} \dist (\theta_s,C_s)^2 \right)ds.
\end{equation}
The integrand of the stochastic integral is in $\Hcal^{2,n}_T$ so the stochastic integral is a martingale. Additionally, the integrand of the drift term is $d\Pbb \otimes dt$ a.s. non-positive, so $R^{(\rho)}$ is a supermartingale. We now seek to use the measurable selection theorem \ref{MeasSelect} to find $\hat\rho$.

\begin{lem}
There exists $\hat\rho \in \Acal_l$ such that $R^{(\hat\rho)}$ is a martingale.
\end{lem}

\begin{proof}
The mapping
\begin{equation}
(\omega,t) \mapsto \Pi_{C_t(\omega)}(\theta_t(\omega)) \equiv \left\{ x \in C_t(\omega): |x - \theta_t(\omega)| = d_t(\omega) \right\},
\end{equation}
where $d = (\dist (\theta_t,C_t))_{t \in [0,T]}$, is a point-to-set mapping between the measurable space $(\Omega \times [0,T],\Pfrak)$ and the Polish space $(\Rbb^n , \Bcal^n)$. Since $C_t(\omega)$ is closed, the set is non-empty for all $(\omega,t)$. The predictability of the processes $\theta$, $\sigma$ and $d$ (by Lemma \ref{distpred}) ensure that the second condition of the measurable selection theorem holds. Thus there exists an $\Rbb^n$-valued predictable process $\hat\rho$ such that $\hat\rho_t(\omega) \in C_t(\omega)$ $\forall (\omega,t)$ and
\begin{equation}
|\hat\rho_t(\omega) - \theta_t(\omega)| = \dist (\theta_t(\omega) , C_t(\omega)) \quad \forall (\omega,t).
\end{equation}
Now recalling $c \in \tilde C$ from before we have that $\Pbb$ a.s.,
\begin{align*}
|\hat\rho_t| &\le |\hat\rho_t - \theta_t| + |\theta_t|\\
 &= \dist (\theta_t , C_t) + |\theta_t|\\
 &\le 2M_2 + K|c|,
\end{align*}
i.e. $\hat\rho$ is $\Pbb$ a.s. bounded uniformly over $t \in [0,T]$. Thus $\hat\rho \in \Hcal^{2,n}_T$. So $\hat\rho \in \Acal_l$. Finally,
\begin{equation}
R^{(\hat\rho)}_t = Y_0 + \int_0^t (\hat\rho_s + Z_s) \cdot dW_s
\end{equation}
so $R^{(\hat\rho)}$ so a martingale.
\end{proof}
It only takes a few additional steps to prove our main theorem:

\begin{thm}
Our optimisation problem is maximised at $\hat\rho \in \Acal_l$ as defined above, with value
\begin{equation}
V(x) = \log(x) + \Ebb \left[ \int_0^T f(s) ds \right]
\end{equation}
where
\begin{equation*}
f(t) = r_t + \frac{1}{2}|\theta_t|^2 - \frac{1}{2} \dist (\theta_t,C_t)^2.
\end{equation*}

\end{thm}

\begin{proof}
Let $\rho \in \Acal_l$. Then by the supermartingale and martingale properties of $R^{(\rho)}$,
\begin{equation*}
\Ebb \left[ R^{(\rho)}_T \right] \le R^{(\rho)}_0 = R^{(\hat\rho)}_0 = \Ebb \left[ R^{(\hat\rho)}_T \right].
\end{equation*}
So
\begin{align*}
V(x) &= \log(x) + \sup_{\rho \in \Acal_l} \Ebb \left[ \int_0^T \rho_s \cdot dW_s + \int_0^T \left(r_s + \rho_s \cdot \theta_s - \frac{1}{2}|\rho_s|^2\right)ds \right]\\
 &= \log(x) + \sup_{\rho \in \Acal_l} \Ebb \left[ R^{(\rho)}_T \right]\\
 &= \log(x) + \Ebb \left[ R^{(\hat\rho)}_T \right]\\
 &= \log(x) + Y_0\\
 &= \log(x) + \Ebb \left[ \int_0^T f(s)ds - \int_0^T Z_s^*dW_s \right]\\
 &= \log(x) + \Ebb \left[ \int_0^T f(s)ds \right],
\end{align*}
where the last equality follows since $Z \in \Hcal^{2,n}_T$ so the stochastic integral is a martingale.
\end{proof}

\begin{rem}
\begin{enumerate}
 \item Suppose we have no set constraint, or equivalently $\tilde C = \Rbb^d$. It follows that $C_t = \sigma_t^* \Rbb^d = \image (\sigma_t^*)$, and by definition $\theta_t \in \image (\sigma_t^*)$ $\forall t \in [0,T]$ $\Pbb$ a.s., so we get the simplified form of the value function
\begin{equation}
\begin{split}
V(x) &= \log(x) + \Ebb \left[ \int_0^T \left( r_s + \frac{1}{2} |\theta_s|^2 \right) ds \right]\\
 &= \log(x) + \frac{1}{2} \lVert \theta \rVert ^2 + \Ebb \left[ \int_0^T r_s ds \right].
\end{split}
\end{equation}
 \item By using trading strategies described by $(X^{(\rho)},\rho)$ we are constraining $X^{(\rho)}$ to be positive. Were we to allow $X$ to be merely bounded below, our derived value of $V(x)$ may be greater, but we would not be able to take advantage of the properties of the stochastic exponential. We have also required that admissible processes $\rho$ be square-integrable, another constraint required for this derivation to work.
\end{enumerate}
\end{rem}

\begin{rem}
The proofs of the analogous results for power and exponential utility are slightly more complicated, but follow the same general idea of defining a family of processes $R$ which has nice martingale and supermartingale properties, and which incorporates an underlying BSDE. These proofs are given in Hu, Imkeller and M\"{u}ller (2005) in the case that the short rate $r$ is zero.
\end{rem}

\end{document}